\documentclass[11pt]{article}

\usepackage[utf8]{inputenc}
\usepackage{amsmath,amsthm,amssymb,enumerate,framed}
\usepackage{color,colortab,epsfig}
\usepackage{graphicx, caption, subcaption}
\usepackage{psfrag}
\usepackage{tikz, float}

\usetikzlibrary{shapes.geometric, arrows, positioning}

\newcommand{\N}{\mathbb{N}}
\newcommand{\Z}{\mathbb{Z}}
\newcommand{\R}{\mathbb{R}}

\newcommand{\intr}{\operatorname{int}}

\newcommand{\chf}{\mathbf{1}}

\renewcommand{\epsilon}{\varepsilon}

\newtheoremstyle{mythmstyle}
	{\topsep}
	{\topsep}
	{\itshape}
	{}
	{\scshape}
	{.}
	{3pt}
	{}
\theoremstyle{mythmstyle}

\newtheorem{nn}{}[section]
\newtheorem{lemma}[nn]{Lemma}
\newtheorem{theorem}[nn]{Theorem}

\newtheorem{prop}[nn]{Proposition}

\newtheorem{REMARK}[nn]{Remark}
\newenvironment{remark}{\begin{REMARK}}{\end{REMARK}}

\newenvironment{cpf}{\begin{trivlist} \item[] {\em Proof of Claim.}}{\hspace*{\stretch{1}} $\diamond$ \end{trivlist}}

\newtheoremstyle{itsemicolon}{}{}{\mdseries\rmfamily}{}{\itshape}{:}{ }{}
\newtheoremstyle{itdot}{}{}{\mdseries\rmfamily}{}{\itshape}{:}{ }{}
\theoremstyle{itdot}
\newtheorem*{msc*}{2010 Mathematics Subject Classification}
\newtheorem*{keywords*}{Keywords}

\numberwithin{equation}{section}

\title{Extreme functions with an arbitrary number of slopes\thanks{Amitabh Basu and Joseph Paat were supported by the NSF grant CMMI1452820. Michele Conforti and Marco Di Summa were supported by the grant ``Progetto di Ateneo 2013" of the University of Padova}}
\author{Amitabh Basu\footnote{Department of Applied Mathematics and Statistics, The Johns Hopkins University}
\and Michele Conforti\footnote{Dipartimento di Matematica, Universit\`a degli Studi di Padova, Italy.}
\and Marco Di Summa\footnotemark[2]
\and Joseph Paat\footnotemark[1]}

\begin{document}
\maketitle

\begin{abstract}
For the one dimensional infinite group relaxation, we construct a sequence of extreme valid functions that are piecewise linear and such that for every natural number $k\geq 2$, there is a function in the sequence with $k$ slopes. This settles an open question in this area regarding a universal bound on the number of slopes for extreme functions. The function which is the pointwise limit of this sequence is an extreme valid function that is continuous and has an infinite number of slopes. This provides a new and more refined counterexample to an old conjecture of Gomory and Johnson stating that all extreme functions are piecewise linear. These constructions are extended to obtain functions for the higher dimensional group problems via the sequential-merge operation of Dey and Richard.
%\keywords{Cutting plane theory \and Infinite group relaxation \and Minimal valid functions \and Extreme valid functions}
%\subclass{90C10 \and 90C11}
\end{abstract}

\section{Introduction}

%{\red Change to $n$ dimensions}

 Let $b\in \R^n\setminus \Z^n$. The {\em infinite group relaxation} $I_b$ is the set of functions $y:\R^n\to \Z_+$ having finite support (that is, $\{r \in \R^n:\;y(r)>0\}$ is a finite set)
 satisfying
 \begin{equation}\label{eq:group-problem}\sum_{r\in \R^n} ry(r)\in b+\Z^n. %\;x_r\in \Z_+.
 \end{equation}
 A function $\pi:\;\R^n\to \R_+$ is \emph{valid} for $I_b$ if
 \begin{equation}\sum_{r\in \R^n} \pi(r)y(r)\ge 1,\mbox{ for every }\;y \in R_b(\R^n,\Z^n).
 \end{equation}
The set $I_b$ has been referred to by multiple names in the literature, see, e.g.,~\cite{basu2016light}.

 Valid functions for the infinite group relaxation were first introduced by Gomory and Johnson~\cite{infinite,infinite2} as means to obtain cutting planes for mixed-integer programs. This idea has recently culminated in the study of {\em cut-generating functions} which has become one of the central aspects of modern cutting plane theory. The surveys of  Basu, Hildebrand, K\"oppe~\cite{basu2016light,basu2016light2} and Basu, Conforti, Di Summa~\cite{basu2015geometric} provide a comprehensive introduction to the subject and survey the recent advances.

 The most well known valid function is the Gomory mixed-integer (GMI) function, which is a valid function for $n=1$. The GMI is defined as follows:

\begin{equation}\label{Gom-funct}
\phi(x) = \begin{cases} \frac{1}{b}x,&0\leq x< b\\  \frac{1}{1-b}-\left(\frac{1}{1-b}\right)x,&b\leq x< 1\\ \phi(x-j),&x\in [j,j+1),~j\in \Z\setminus\{0\}~.\end{cases}
\end{equation}

  A valid function $\pi$ is \emph{minimal} if $\pi=\pi'$ for every valid function $\pi'$ such that $\pi'\le \pi$. The motivation for this definition is the following. Given valid functions $\pi$ and $\pi'$, we say that $\pi'$ \emph{dominates} $\pi$ if for every function $y:\R^n\to \Z_+$ with finite support satisfying the inequality $\sum_{r\in \R^n} \pi(r)y(r)< 1$, the function $y$ also satisfies the inequality $\sum_{r\in \R^n} \pi'(r)y(r)< 1$.  Observe that if $\pi'$ dominates $\pi$, then $\pi$ is redundant for describing $I_b$. Furthermore, if $\pi'\leq \pi$, then $\pi'$ dominates $\pi$. Thus, if a valid function is not minimal, then it is redundant for describing $I_b$.
  
 %Let $\mathcal{N}$ be the set of nonnegative functions from $\R^n$ to $\R$ with finite support. Given valid functions  $\pi$ and $\pi'$ such that $\pi'\le \pi$ and  $\pi\ne\pi'$, it holds that $\{y \in \mathcal{N}:\sum \pi'(r)y(r)\ge 1\}\subsetneq \{y \in \mathcal{N}:\sum \pi(r)y(r)\ge 1\}$. Therefore if a valid function is not minimal, then it is redundant.

  A function $\theta\colon \R^n \to \R$ is {\em subadditive} if $\theta(x)+\theta(y) \geq \theta(x+y)$  for all $x,y \in \R^n$. $\theta$  satisfies the \emph{symmetry condition} if $\theta(x) + \theta(b - x) = 1$ for all $x\in \R^n$. Finally, $\theta$ is \emph{periodic modulo
  $\Z^n$} if $\theta(x) = \theta(x + z)$ for all  $x \in \R^n$ and $z \in \Z^n$.

\begin{theorem}[Gomory and Johnson \cite{infinite}] \label{thm:minimalinteger} A function
  $\pi \colon \R^n \to \R_+$ is a minimal valid function for $I_b$ if and only if $\pi(z) = 0$ for
  all $z\in \Z^n$, $\pi$ is subadditive, and $\pi$ satisfies the symmetry
  condition. (These conditions imply that $\pi$ is periodic modulo $\Z^n$
  and $\pi(b)=1$.)
\end{theorem}

It is easy to check that the Gomory mixed-integer function defined above is subadditive and satisfies the symmetry condition. Therefore, by the above theorem, it is a minimal function.

Minimal functions are the ones that are not dominated by any other function. However a minimal function may be implied by the convex combinations of other valid functions. Gomory and Johnson define
 a valid function  $\pi$ to be \emph{extreme} if $\pi=\pi_1=\pi_2$ for every pair of valid functions $\pi_1,\pi_2$ such that $\pi = \frac{\pi_1+\pi_2}{2}$. If $\pi$ is a valid function which is extreme, then $\pi$ is easily seen to be minimal. Therefore extremality is a stronger requirement. An even more stringent definition is that of a {\em facet}. For any valid function $\pi$, define $P(\pi) := \{y \in R_b(\R^n,\Z^n) : \sum_{r\in \R^n} \pi(r)y(r) = 1\}$. A valid function $\pi$ is a {\em facet} if $P(\pi) \subseteq P(\pi')$ implies $\pi = \pi'$ for all valid functions $\pi'$. It can be verified that every facet is extreme~\cite[Lemma 1.3]{basu-hildebrand-koeppe-molinaro:k+1-slope}. It was recently shown that continuous piecewise linear extreme functions are also facets; however, there exist discontinuous piecewise linear extreme functions which are not facets~\cite{koppe2016notions}.%It is not known whether every extreme function is a facet.

We will need a formal notion of piecewise linear functions which we introduce now. A {\em regular polyhedral complex in $\R^n$} is a collection of polyhedra $P_j$, $j\in J$ such that three conditions are satisfied: 1) $\R^n = \bigcup_{j\in J}P_j$, 2) for any $i,j\in J$, $P_i \cap P_j$ is a common face of $P_i$ and $P_j$ and also belongs to the collection, and 3) any bounded subset of $\R^n$ intersects only finitely many polyhedra from the collection. We say a  function $\theta:\R^n\to \R$ is \emph{piecewise linear} if there is a regular polyhedral complex in $\R^n$ such that $\theta$ is affine linear over the interior of each polyhedron in the complex. Note this definition allows for discontinuous piecewise linear functions. For a natural number $k$, we say that a piecewise linear function has $k$ slopes if it has $k$ distinct values for the gradient, where it exists.

%{\red Switch to $n=1$}

\begin{theorem}[Gomory and Johnson \cite{infinite}] \label{thm:2-slope} Let
  $\pi \colon \R \to \R_+$ be a minimal valid function which is continuous, piecewise linear and has only 2 slopes. Then $\pi$ is a facet (and therefore extreme).
  \end{theorem}

\noindent In particular, the above theorem implies that the Gomory mixed-integer function is a facet.

\medskip

For the one-dimensional problem, i.e., $n=1$, extreme valid functions or facets that are piecewise linear and have few slopes received the largest number of hits in the shooting experiments~\cite{GJE2002} and seem to be the most useful in practice. Indeed
Gomory and Johnson \cite{tspace} conjectured that every valid function that is extreme is piecewise linear. This has been disproved by Basu et al.~\cite{bccz08222222}.

Minimal valid functions with 3 slopes are not always extreme. However, Gomory and Johnson constructed an extreme function that is piecewise linear with 3 slopes. It appears to be hard to construct extreme functions that are piecewise linear with many slopes. Indeed, until 2013, all known families of piecewise linear extreme functions had at most 4 slopes. This had led Dey and Richard to pose the question of constructing extreme functions with more than 4 slopes at a 2010 Aussois meeting~\cite{santanu-presentation}. In 2013, Hildebrand, in an unpublished result, constructed an extreme function that is piecewise linear with 5 slopes and very recently K\"oppe and Zhou~\cite{koeppe-zhou:extreme-search} constructed an extreme function that is piecewise linear with 28 slopes. These functions were found through a clever computer search.

K\"oppe and Zhou~\cite{koeppe-zhou:extreme-search} expressed
 the belief that there exist extreme functions that are piecewise linear and have an arbitrary number of slopes (this is also stated as Open Question 2.15 in the survey by Basu, Hildebrand and K\"oppe~\cite{basu2016light}.) We prove this.
More precisely, we show the following:
\begin{theorem}\label{thm:kslopes} Let $b\in \R\setminus \Z$. For $k\geq 2$, there exists a facet (and therefore an extreme valid function) for $I_b$ that is piecewise linear with $k$ slopes.
\end{theorem}

 The proof of Theorem~\ref{thm:kslopes} provided here is constructive.
We  define a sequence of functions $\{\pi_k\}_{k=2}^\infty$, where $\pi_2$ is the Gomory mixed-integer function, and $\pi_3$ is an instantiation of a construction of extreme functions that are piecewise linear and have 3 slopes provided by Gomory and Johnson~\cite{tspace}. We first prove  some properties about each  function $\pi_k$ in Section \ref{sect-constr}. In Section \ref{sect-minimal} we use these properties to show that these functions are subadditive and satisfy the symmetry condition. Therefore each  function $\pi_k$ is a minimal valid function, as it  satisfies the conditions of Theorem \ref{thm:minimalinteger}. Section \ref{sect-extreme} is devoted to the proof that each  function $\pi_k$ is a facet.

Our next result states that the function which is the pointwise limit of this sequence is an extreme function that is continuous and has an infinite number of slopes. The proof appears in Section \ref{sec:proof-infinite-slopes}.

\begin{theorem}\label{cor:infinite_slopes} Let $b\in \R\setminus \Z$. There exists a continuous function $\pi_{\infty}$ that is a facet (and therefore extreme) for $I_b$ with an infinite number of slopes (i.e., values for the derivative of $\pi_\infty$).
\end{theorem}

This also provides a different family of counterexamples to the Gomory-Johnson conjecture that all extreme functions are piecewise linear. In contrast, the previous family of counterexamples from~\cite{bccz08222222} all have 2 slopes. 

Note that in Theorems~\ref{thm:kslopes} and~\ref{cor:infinite_slopes}, we may assume $b\in (0,1)$ since extreme functions are periodic with respect to $\Z$. We give constructions to establish Theorems~\ref{thm:kslopes} and~\ref{cor:infinite_slopes} with $b$ in the interval $(0, \frac{1}{2}]$. One may obtain extreme functions for values of $b\in[\frac{1}{2}, 1)$ by reflecting the constructions about $0$. This is an example of an automorphism introduced by Gomory and later used by Johnson (Theorem 8.2 in~\cite{johnson1974group}, see also Theorem~\ref{thm:reflection} in the Appendix).

%Indeed, one can check that $\pi$ is minimal/extreme/facet for $R_b(\R, \Z)$ when $b\in (0,1/2]$ if and only if $\tilde\pi:\R\to \R$ defined by $\tilde\pi(x) :=  \pi(-x)$ is minimal/extreme/facet for $R_{1-b}(\R, \Z)$, respectively -- see Theorem~\ref{thm:reflection} in the Appendix.

We end the paper by using the {\em sequential-merge} operation invented by Dey and Richard~\cite{dey2} to construct facets for the $n$-dimensional infinite group relaxation (for any $n \geq 1$) with an arbitrary number of slopes. The idea is to use the sequential-merge operation iteratively on the facets constructed for Theorem~\ref{thm:kslopes} and the GMI function from~\eqref{Gom-funct}. See Theorem~\ref{thm:seq_merge} for a detailed statement.

\section{A Construction of k-Slope Functions $\pi_k$}\label{sect-constr}
Let $b\in (0, \frac{1}{2}]$. Let $\pi_2$ be the Gomory mixed-integer function defined by \eqref{Gom-funct}.

\iffalse
\begin{equation*}
\pi_2(x) = \begin{cases} \frac{1}{b}x,&0\leq x< b\\  \frac{1}{1-b}-\frac{1}{1-b}x,&b\leq x< 1\\ \pi_2(x-j),&x\in [j,j+1),~j\in \Z\setminus\{0\}\end{cases}.
\end{equation*}

Observe that $\pi_2$ is the famous Gomory Mixed Integer Cut corresponding to $b$ (see~\cite{gomory1960algorithm}).
\fi

In constructing $\pi_k$ for $k\geq3$, we use the following intervals:
\begin{equation*}
\begin{array}{llll}
I^k_1:= [0, b(\frac{1}{8})^{k-2}], & I^k_2:=& [b(\frac{1}{8})^{k-2}, 2b(\frac{1}{8})^{k-2}],\\\\
I^k_3:= [2b(\frac{1}{8})^{k-2}, b-2b(\frac{1}{8})^{k-2}], & I^k_4 :=& [b-2b(\frac{1}{8})^{k-2}, b-b(\frac{1}{8})^{k-2}],\\\\
I^k_5:= [b-b(\frac{1}{8})^{k-2},b],& I^k_6:=& [b, 1).
\end{array}
\end{equation*}
Given $\pi_{k-1}$, where $k-1\geq 2$, define $\pi_{k}$ to be
\begin{equation*}
\pi_{k}(x) = \begin{cases}
\left(\frac{2^{k-2}-b}{b-b^2}\right)x,&x\in I^k_1\\
\frac{4^{2-k}}{1-b}-\left(\frac{1}{1-b}\right)x,&  x\in I^k_2\\
\frac{1-4^{2-k}}{1-b}-\left(\frac{1}{1-b}\right)x,&x\in I^k_4\\
\frac{1-2^{k-2}}{1-b}+\left(\frac{2^{k-2}-b}{b-b^2}\right)x ,& x\in I^k_5\\
\pi_{k-1}(x),&x\in I^k_3 \cup I^k_6\\
\\ \pi_{k}(x-j),&x\in [j,j+1),~j\in \Z\setminus\{0\}~.
\end{cases}
\end{equation*}

\begin{prop}\label{prop:nonnegativity}
Let $k\geq 2$. Then $\pi_k$ is well-defined, continuous, nonnegative, and $\pi_k(x)=0$ if and only if $x\in \Z$.
\end{prop}
The proof of Proposition~\ref{prop:nonnegativity} is in the Appendix. 

Figure \ref{fig:kslopes} shows $\pi_k$ for various values of $k$ when $b=\frac{1}{2}$. The plots were generated using the help of a software package created by Hong, K\"{o}ppe, and Zhou \cite{sage_hkz}.

\begin{figure}
\centering
\begin{tabular}{cc}
%Switch to the following for black and white
%\includegraphics[width=.5\textwidth]{Figures/2_slope.eps} &
%\includegraphics[width=.5\textwidth]{Figures/3_slope.eps}\\
%(a) k=2 & (b) k=3\\
%\includegraphics[width=.5\textwidth]{Figures/4_slope.eps}&\\
%(c) k = 4 &

%Switch to the following for color
\includegraphics[width=.5\textwidth]{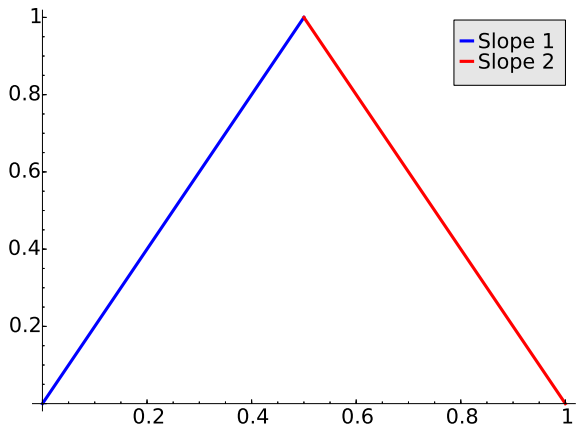} &
\includegraphics[width=.5\textwidth]{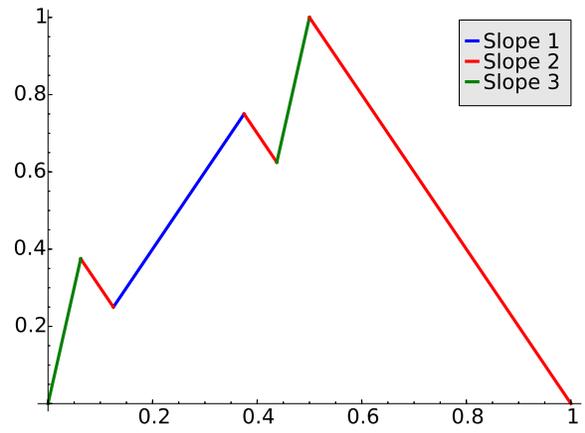}\\
(a) k=2 & (b) k=3\\
\includegraphics[width=.5\textwidth]{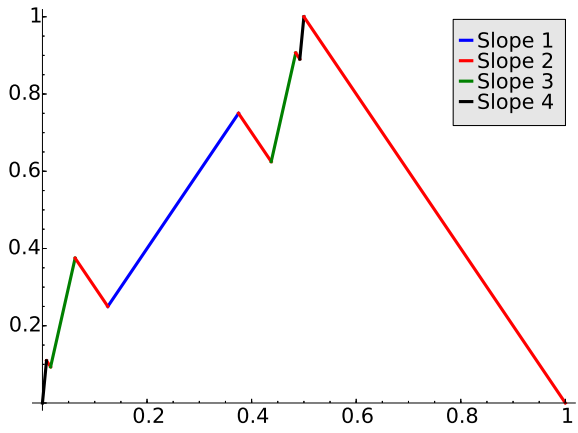}&\\
(c) k = 4 &
\end{tabular}
\caption{Plots of $\pi_k$ for $k\in \{1,2,3\}$ and $b = \frac{1}{2}$.}
\label{fig:kslopes}
\end{figure}

Observe that $\pi_k$ is built recursively with the Gomory mixed-integer function as the base case. Intuitively, $\pi_k$ is created by adding  to $\pi_{k-1}$  a perturbation on a small interval to the right of $0$ and applying a symmetric perturbation on an interval to the left of $b$; the interval $[b,1)$ is kept intact. These small perturbations allow $\pi_k$ to maintain much of the structure of $\pi_{k-1}$, but the number of distinct slopes is increased by one.
We collect some useful properties of $\pi_k$ in Propositions~\ref{prop:properties} and \ref{prop:slopes}.

%The following is a collection of useful properties about $\pi_k$ when $k\geq 3$.

\begin{prop}\label{prop:properties} Let $k\geq 3$. Then
\begin{enumerate}[(i)]
\item $I^k_1\cup I^k_2 \subsetneq I^{k-1}_1$ and $I^k_4\cup I^k_5 \subsetneq I^{k-1}_5$.
\item If $x\in I^k_3\cup I^k_6$, then $\pi_k(x) = \pi_{k-1}(x)$. If $x\in I^k_1\cup I^k_2$, then $\pi_k(x) \geq \pi_{k-1}(x)$. If $x\in I^k_4\cup I^k_5$, then $\pi_k(x) \leq \pi_{k-1}(x)$.
%\item $-\pi_k$ is convex on $I^k_1\cup I^k_2$ and $\pi_k$ is convex on $I^k_4\cup I^k_5$.
%\item[(iv)] Let $y\in I^k_4\cup I^k_5$ such that $y\neq b$ and take $x\in [0,b-y]$. Then $\pi_k(x+y)\leq \pi_k(y)+\left(\frac{1-\pi_k(y)}{b-y}\right)x$. Also, $\left(\frac{\pi_k(b-y)}{b-y}\right)x \leq \pi_k(x)$.
\item For any $x\in (0,1)\setminus\{b\}$, there exists some natural number $N_x$ such that $x \in I^{N_x}_3\cup I^{N_x}_6$ and $\pi_{k}(x) = \pi_{N_x}(x)$ whenever $k\geq N_x$.
\end{enumerate}
\end{prop}

\begin{proof}~

\begin{enumerate}[$(i)$]
\item Observe that $$b\left(\frac{1}{8}\right)^{k-3} = 8b\left(\frac{1}{8}\right)^{k-2}> 2b\left(\frac{1}{8}\right)^{k-2}~.$$
By the definitions of $I^k_1, I^k_2$ and $I^{k-1}_1$, it follows that $I^k_1\cup I^k_2 \subsetneq I^{k-1}_1$. A similar argument shows that $I^k_4\cup I^k_5 \subsetneq I^{k-1}_5$.

\vspace{.1 in}
\item Let $x\in [0,1)$. If $x\in  I^k_3\cup I^k_6$, then $\pi_k(x) = \pi_{k-1}(x)$ by definition. If $x\in I^k_1$, then from $(i)$ it follows that $x\in I^{k-1}_1$. Note that $$\left(\frac{2^{k-2}-b}{b-b^2}\right)x \geq \left(\frac{2^{k-3}-b}{b-b^2}\right)x~,$$ and so $\pi_k(x)\geq \pi_{k-1}(x)$. If $x\in I^k_2$, then again from $(i)$, $x\in I^{k-1}_1$ and it follows that
\begin{equation*}
\begin{array}{rcl}
\frac{4^{2-k}}{1-b}-\left(\frac{1}{1-b}\right)x & = & \left(\frac{1}{1-b}\right)\left(4^{2-k}-x\right) \\
& \geq &\left(\frac{1}{1-b}\right)\left(4^{2-k} - 2b\left(\frac{1}{8}\right)^{k-2}\right)\qquad \text{since }x\in I^k_2\\
& = & \left(\frac{1}{b-b^2}\right)\left(2^{k-3}\left(2b\left(\frac{1}{8}\right)^{k-2}\right)-b\left(2b\left(\frac{1}{8}\right)^{k-2}\right)\right)\\
& \geq &\left(\frac{2^{k-3}-b}{b-b^2}\right)x \quad\qquad\qquad\qquad\qquad\text{since }x\in I^k_2.
\end{array}
\end{equation*}

Hence $\pi_k(x)\geq \pi_{k-1}(x)$ on $I_1^k\cup I^k_2$. A similar argument shows that $\pi_k(x)\leq \pi_{k-1}(x)$ on $I_4^k\cup I^k_5$.

%\item By definition, $\pi_k$ is affine linear over $I^k_1$ with positive slope and affine linear over $I^k_2$ with negative slope. Since $\pi_k$ is continuous, it is therefore concave. So $-\pi_k$ is a convex function over $I^k_1\cup I^k_2$. The same argument shows that $\pi_k$ is convex over $I^k_4\cup I^k_5$.

%\item Fix $y\in  I^k_4\cup I^k_5\setminus\{b\}$. It follows by assumption that $x+y\in [y,b]$. Therefore $\lambda = \frac{b-x-y}{b-y} \in [0,1]$. Using the facts that $\pi_k$ is convex over $[y,b]$ from (iii) and $\pi_k(b)=1$, we obtain
%$$\pi_k(x+y) = \pi_k(\lambda y+(1-\lambda)b)\leq \lambda\pi_k(y)+(1-\lambda)\pi_k(b) = \pi_k(y)+\left(\frac{1-\pi_k(y)}{b-y}\right)x.$$
%
%The other inequality follows from the fact that $-\pi_k$ is convex over $I^k_1\cup I^k_2$ by (iii).

\item
Notice that $I^k_3\subseteq I^{k+1}_3$ for every natural number $k$, and as $k\to \infty$, %the right endpoint of $I_1^k\cup I_2^k$ converges to $0$ and the left endpoint of $I_4^k\cup I_5^k$ converges to $b$. Equivalently,
$I^k_3$ converges to $(0,b)$. Thus, there exists some natural number $N_x$ such that $x \in I^{k}_3\cup I^{k}_6$ for every natural number $k\geq N_x$. By the definition of $\pi_{k}$, if $k\geq N_x$, then $\pi_k(x) = \pi_{N_x}(x)$. \qed %Combining the definition of $\pi_k$ and (i) above, the result follows.
\end{enumerate}
\end{proof}%End properties

\begin{prop}\label{prop:slopes} For each integer $k\geq 2$, the function $\pi_k$ is piecewise linear and has $k$ slopes taking values $-\frac{1}{1-b}$ and $\{\frac{2^{i-2}-b}{b-b^2}\}_{i=2}^k$. Moreover, if $k\geq 3$, then $\pi_k$ has the $k-2$ slopes $\{\frac{2^{i-2}-b}{b-b^2}\}_{i=2}^{k-1}$ on $I^k_3$ and the slope $-\frac{1}{1-b}$ on $I^k_6$.
\end{prop}
\begin{proof}
We proceed by induction. For $\pi_2$ and $\pi_3$, the result is readily verified by the definitions. So, assume that for $k-1\geq 3$, $\pi_{k-1}$ is piecewise linear with $k-1$ slopes and has the $k-3$ slopes $\{\frac{2^{i-2}-b}{b-b^2}\}_{i=2}^{k-2}$ on $I^{k-1}_3$ and the slope $-\frac{1}{1-b}$ on $I^{k-1}_6$. %, {\blue with $k-1$ slopes on $I^{k-1}_3\cup I^{k-1}_6$.}

%Observe that for each value of $j$, $\pi_j$ has a slope of $-\frac{1}{1-b}$ on the interval $(b,1)$. Therefore on the interval $[0,b)$, the function $\pi_{k-1}$ must take on slope values $\{\frac{2^{i-2}-b}{b-b^2}\}_{i=2}^{k-1}$ ($\pi_{k-1}$ also admits a slope of $-\frac{1}{1-b}$ on some subintervals contained in $[0, b)$). 
Consider $\pi_k$. The fact that $\pi_k$ is piecewise linear follows from the definition of $\pi_k$, and the induction hypothesis that $\pi_{k-1}$ is piecewise linear.%, and the fact that $\pi_k$ is continuous by Proposition~\ref{prop:nonnegativity}. 

It is left to consider the slope values of $\pi_k$. By Proposition \ref{prop:properties} $(ii)$, $\pi_k=\pi_{k-1}$ everywhere except $I^k_1\cup I^k_2$ and $I^k_4\cup I^k_5$, on which $\pi_k$ takes on slope values $\frac{2^{k-2}-b}{b-b^2}$ and $-\frac{1}{1-b}$ by definition. Since $I^k_1\cup I^k_2\subsetneq I^{k-1}_1$ and $I^k_4\cup I^k_5\subsetneq I^{k-1}_5$ by Proposition \ref{prop:properties} $(i)$, it follows from the induction hypothesis that $\pi_k$ also has slopes taking values $\{\frac{2^{i-2}-b}{b-b^2}\}_{i=2}^{k-1}$. Thus, $\pi_k$ takes on slopes values $\{\frac{2^{i-2}-b}{b-b^2}\}_{i=2}^{k}$ and $-\frac{1}{1-b}$.

%It is left to show that $\pi_k$ is piecewise linear. By Proposition \ref{prop:properties} $(ii)$ and the induction hypothesis, it is sufficient to show that $\pi_k$ is piecewise linear on $I^k_1\cup I^k_2$ and $I^k_4\cup I^k_5$, and that $\pi_k\left(2b\left(\frac{1}{8}\right)^{k-2}\right)= \pi_{k-1}\left(2b\left(\frac{1}{8}\right)^{k-2}\right)$ and $\pi_k\left(b-2b(\frac{1}{8})^{k-2}\right) = \pi_{k-1}\left(b-2b\left(\frac{1}{8}\right)^{k-2}\right)$. Note that $\pi_k$ is piecewise linear on $I^k_1\cup I^k_2$ and $I^k_4\cup I^k_5$ by definition. It is straightforward to check that
%$\pi_k\left(2b\left(\frac{1}{8}\right)^{k-2}\right) =\pi_{k-1}\left(2b\left(\frac{1}{8}\right)^{k-2}\right) $ and
%$\pi_k\left(b-2b\left(\frac{1}{8}\right)^{k-2}\right)  = \pi_{k-1}\left(b-2b\left(\frac{1}{8}\right)^{k-2}\right)$.
%Thus $\pi_k$ is piecewise linear, as desired.

Finally, by definition, $\pi_k = \pi_{k-1}$ on $I^k_3\cup I^k_6$. Moreover, since $\pi_k$ has slope values  $\{\frac{2^{i-2}-b}{b-b^2}\}_{i=2}^k$, and $\pi_{k-1}$ has slope values $\{\frac{2^{i-2}-b}{b-b^2}\}_{i=2}^{k-1}$, the only new slope in $\pi_k$ is $\frac{2^{k-2}-b}{b-b^2}$, which only appears on $I^k_1\cup I^k_5$. Thus, $\pi_k$ has the $k-2$ slopes $\{\frac{2^{i-2}-b}{b-b^2}\}_{i=2}^{k-1}$ on $I^k_3$ and the slope $-\frac{1}{1-b}$ on $I^k_6$.\qed
\end{proof}%End k-slopes and pwl

\section{Proof of Minimality of $\pi_k$}\label{sect-minimal}

In the proof of Theorem~\ref{thm:kslopes}, we will show that $\pi_k$ is a facet using the so-called Facet Theorem -- see Theorem~\ref{thm:facet} in Section~\ref{sect-extreme}. Applying the Facet Theorem to $\pi_k$ requires that $\pi_k$ be a minimal valid function for $I_b$, which we verify in this section. Since by definition $\pi_k$ is nonnegative, $\pi_k(0)=0$, and $\pi_k$ is periodic, by Theorem \ref{thm:minimalinteger}, it is sufficient to show that  (a) $\pi_k(x) = \pi_k(b-x)$ for all $x\in [0,1)$, i.e. that $\pi_k$ satisfies the symmetry condition, and (b) $\pi_k$ is subadditive. We show (a) and (b) in Propositions \ref{prop:symmetric} and \ref{prop:subadditive}, respectively.

\begin{prop}\label{prop:symmetric} $\pi_k$ satisfies the symmetry condition for all $k \geq 2$. \end{prop}
\begin{proof}
We proceed by induction on $k$. The Gomory mixed-integer function is known to be minimal, and hence $\pi_2$ is symmetric. Assume $\pi_{k-1}$ satisfies the symmetry condition for $k-1\geq 2$ and consider $x\in [0,1)$. Observe that $x\in I^k_1$ if and only if $b-x\in I^k_5$.
Therefore if $x\in I^k_1$, then $$\pi_k(x)+\pi_k(b-x) =\left(\frac{2^{k-2}-b}{b-b^2}\right)x+\frac{1-2^{k-2}}{1-b}+\left(\frac{2^{k-2}-b}{b-b^2}\right)(b-x) = 1~.$$
A similar argument can be used to show that $\pi_k$ satisfies the symmetry condition on the intervals $I^k_2$ and $I^k_4$. If $x\not\in I_1^k\cup I_2^k\cup I^k_4\cup I^k_5$, then $b-x\not\in I_1^k\cup I_2^k\cup I^k_4\cup I^k_5$, and so symmetry holds by induction.\qed
\end{proof}%End Symmetry

\begin{prop}\label{prop:subadditive} $\pi_k$ is subadditive for all $k \geq2$. \end{prop}
\begin{proof}
We proceed by induction on $k$. Note that $\pi_2$ is subadditive, so assume $\pi_{k-1}$ is subadditive for $k-1\geq 2$. By periodicity of $\pi_k$, it suffices to check $\pi_k(x) + \pi_k(y) \geq \pi_k(x+y)$  for all $x,y\in [0,1)$ and $x\leq y$.
\medskip

\noindent{\sc Claim.}
\emph{If $y\in I^k_6 = [b,1)$, then $\pi_k(x+y)\leq \pi_k(x)+\pi_k(y)$.}

\begin{cpf} Since $\pi_k$ is piecewise linear and continuous by Propositions~\ref{prop:nonnegativity} and~\ref{prop:slopes}, we may integrate it over any bounded domain. %Let $\pi'_k$ denote the derivative of $\pi_k$ (where defined). 
A direct calculation shows
\begin{equation*}
\begin{array}{rlr}
\pi_k(x+y) = & \pi_k(x+(y-1)) & \text{by periodicity of }\pi_k\\\\
=& \pi_k(x)+\int_x^{x-(1-y)}\pi_k'(t)dt&\\\\
= &\pi_k(x)+\int_{x-(1-y)}^x-\pi_k'(t)dt&\\\\
\leq& \pi_k(x)+\int_{y}^1-\pi_k'(t)dt&\\\\
= &\pi_k(x)-\pi_k(1)+\pi_k(y)&\\
= &\pi_k(x)+\pi_k(y) &\text{since }\pi_k(1)=0.
\end{array}
\end{equation*}
The inequality follows from Proposition \ref{prop:slopes}, as the minimum value of the slope for $\pi_k$ is $-\frac{1}{1-b}$ and this is the slope over the interval $[b, 1] \supseteq [y,1]$. This concludes the proof of the claim.\end{cpf}

 By the above claim, it suffices to consider the case $y<b$. Since $b \leq \frac12$, this implies that $x\le y\le x+y < 1$.

\vspace{.1 in}
\textbf{Case 1:} $x+y\in I^k_1\cup  I^k_2$. Note that the derivative $\pi'_k$ is nonincreasing on $I^k_1\cup  I^k_2 \setminus \{b(\frac{1}{8})^{k-2}\}$. Thus,
$$\pi_k(x+y) =  \pi_k(x)+\int_x^{x+y} \pi'_k(t)dt \leq  \pi_k(x)+\int_0^{y} \pi'_k(t)dt = \pi_k(x)+\pi_k(y).$$

%By Proposition \ref{prop:properties} (iii), the function $-\pi_k$ is convex over $I^k_1\cup  I^k_2$. Therefore $\pi_k(x) + \pi_k(y) \geq \pi_k(x+y)$.

\textbf{Case 2:}  $x+y\in I^k_3$. Since  $x,y\in I^k_1\cup I^k_2\cup I^k_3$ we have that  $$\pi_k(x)+\pi_k(y) \geq \pi_{k-1}(x)+\pi_{k-1}(y)  \geq \pi_{k-1}(x+y) = \pi_k(x+y)~,$$
where the first inequality comes from Proposition \ref{prop:properties} $(ii)$, the second inequality comes from the induction hypothesis, and the final inequality comes again from Proposition \ref{prop:properties} $(ii)$.

\textbf{Case 3:}  $x+y\in I^k_4 \cup I^k_5$. If $y\in I^k_1\cup I^k_2\cup I^k_3$, then using the induction hypothesis and Proposition \ref{prop:properties} $(ii)$, it follows that $$\pi_k(x)+\pi_k(y) \geq \pi_{k-1}(x)+\pi_{k-1}(y)  \geq \pi_{k-1}(x+y) \geq \pi_k(x+y).$$ If $y\in I^k_4 \cup I^k_5$, then $x\in [0, b-y]$ and $b-y\in I^k_1\cup I^k_2$. Hence, $x \in I^k_1 \cup I^k_2$. Also, $b-(x+y)\in I^k_1 \cup I^k_2$ since $x+y\in I^k_4 \cup I^k_5$. Thus, we can apply Case 1 to the values $x$ and $b-(x+y)$ to obtain $\pi_k(b-y)\leq \pi_k(b-(x+y))+\pi_k(x)$. Using this, we see that
\begin{equation*}
\begin{array}{rlcl}
\pi_k(x+y) = & 1-\pi_k(b-(x+y))&&\text{by the symmetry property}\\
\leq & 1-\pi_k(b-y)+\pi_k(x)&&\\
= & \pi_k(y)+\pi_k(x)&&\text{by the symmetry property}.
\end{array}
\end{equation*}

\textbf{Case 4:}  $x+y\in I^k_6$. $\pi_k$ has a slope of $-\frac{1}{1-b}$ on the interval $[b,x+y]$. Moreover, by Proposition \ref{prop:slopes}, this is the minimum slope that $\pi_k$ admits. Therefore,

\begin{equation*}
\begin{array}{rcl}
\pi_k(x+y) & = & \pi_k(b) + \int_{b}^{x+y}\pi_k'(t)dt\\
 &\leq&1 + \int_{b-x}^{y}\pi_k'(t)dt\\
& = &1 +(\pi_k(y)-\pi_k(b-x))\\
& = &\pi_k(x)+\pi_k(y)~,
\end{array}
\end{equation*}
where the last equality follows by the symmetry of $\pi_k$.\qed
\end{proof}%End subadditivity

\section{$\pi_k$ is a facet}\label{sect-extreme}

By Proposition \ref{prop:slopes}, in order to prove Theorem \ref{thm:kslopes} it suffices to show the following result.

\begin{prop}\label{prop:pi_k_facet}
$\pi_k$ is a facet for each $k\geq 2$.
\end{prop}

We dedicate the remainder of the section to proving Proposition~\ref{prop:pi_k_facet}. To this end, given a function $\theta: \R^n \to \R$, define
\begin{equation}\label{eq:equality}E(\theta) = \left\{(x, y)\in \R^n \times \R^n: \theta(x)+\theta(y) = \theta(x+y)\right\}.\end{equation}
Our proof of Proposition \ref{prop:pi_k_facet} is based on
the \textit{Facet Theorem}, which gives  a sufficient condition for a function to be a facet~\cite{basu-hildebrand-koeppe-molinaro:k+1-slope,infinite}, and the \textit{Interval Lemma}, which first appeared in~\cite{tspace}, and was subsequently elaborated upon in~\cite{dey1,dey2,dey3,bhk-IPCOext}; see also the survey~\cite{basu2016light,basu2016light2}.

\begin{theorem}[Facet Theorem]\label{thm:facet}
Let $\pi:\R^n\to \R_+$ be a minimal valid function for $I_b$ for some $b \in \R^n\setminus \Z^n$. Suppose that for every minimal function $\theta:\R^n\to \R_+$ satisfying $E(\pi)\subseteq E(\theta)$, it follows that $\theta= \pi$. Then $\pi$ is a facet.
\end{theorem}

\begin{lemma}[Interval Lemma]\label{lem:interval-lemma} Let $U, V$ be nondegenerate closed intervals in $\R$. If $\theta: \R \to \R$ is bounded over $U$ and $V$, and $U\times V\subseteq E(\theta)$, then $\theta$ is affine over $U, V$ and $U+V$ with the same slope.
\end{lemma}
We will often use the above lemma when $\theta$ is a minimal valid function. In this case $\theta$ is bounded, as $0\le \theta \le 1$. We also say a function $\theta:\R \to \R$ is {\em locally bounded} if it is bounded on every compact interval.

\begin{prop}\label{obs:periodicity}
Let $\theta : \R \to \R_+$ be such that $\theta(0) = 0$ and $\theta(x+z) = \theta(x) + \theta(z)$ for all $x \in \R$ and $z \in \Z$. Then $\theta$ is periodic, i.e., $\theta(x+z) = \theta(x)$ for all $x \in \R$ and $z \in \Z$.
\end{prop}

\begin{proof}
It suffices to show that $\theta(z) = 0$ for all $z\in \Z$. This is true since $0 = \theta(0) = \theta(-z) + \theta(z)$ for all $z\in \Z$ and $\theta$ is nonnegative.\qed
\end{proof}

In the following Propositions~\ref{claim:Ik6_general}, \ref{claim:i33_general}, \ref{claim:ik2_induction_general}, \ref{claim:Ij1_induction_general}, we develop some tools towards proving facetness.

%The following claims hold for any $\pi_k$, $k\geq 3$.

 \begin{prop}\label{claim:Ik6_general} Let $k \geq 3$ and let $\pi$ be a minimal valid function such that $\pi = \pi_k$ on $I^k_6$. Then for all locally bounded functions $\theta:\R \to \R_+$ such that $E(\pi) \subseteq E(\theta)$ satisfying $\theta(0)=0, \theta(b) = 1$, we must have $\theta= \pi = \pi_k$ on $I^k_6\cup \{1\}$.\end{prop}
\begin{proof}
Note that $I^k_6\cup \{1\} \equiv [\frac{1+b}{2},1]+[\frac{1+b}{2},1]$ (modulo 1). Since $\pi$ is minimal, Theorem~\ref{thm:minimalinteger} implies $\pi$ is periodic. Since $E(\pi) \subseteq E(\theta)$, Proposition~\ref{obs:periodicity} shows that $\theta$ is periodic. In particular, $\theta(1)=0 =\pi(1)$ and $\theta(b) = 1 =\pi(b)$. Hence $\pi= \theta$ on the endpoints of $I^k_6\cup\{1\}$. Moreover, $x,y\in  [\frac{1+b}{2},1]$ implies that
\begin{equation*}
\begin{array}{rlr}
\pi(x)+\pi(y) = &\left(\frac{1}{1-b}-\left(\frac{1}{1-b}\right)x\right)+\left(\frac{1}{1-b}-\left(\frac{1}{1-b}\right)y\right)&\text{since }\pi = \pi_k\text{ on }I^k_6\\
= &\frac{1}{1-b}-\left(\frac{1}{1-b}\right)\left(x+y-1\right) &\\
= & \pi(x+y-1) & \\
= & \pi(x+y)&\text{by periodicity.}
\end{array}
\end{equation*}
Hence $[\frac{1+b}{2},1]\times [\frac{1+b}{2},1]\subseteq E(\pi) \subseteq E(\theta)$. Lemma~\ref{lem:interval-lemma} then implies that $\theta$ is affine over $I^k_6 \cup \{1\}$. Since $\pi$ is also affine over $I^k_6 \cup \{1\}$ and $\pi= \theta$ on the endpoints of $I^k_6\cup\{1\}$, we must have $\pi = \theta$ on $I^k_6\cup\{1\}$.\qed
\end{proof}

\begin{prop}\label{claim:i33_general} Let $k \geq 3$ and let $\pi$ be a minimal valid function such that $\pi = \pi_k$ on $I^3_3$. Then for all locally bounded functions $\theta:\R \to \R_+$ such that $E(\pi) \subseteq E(\theta)$ satisfying $\theta(\frac{b}{2})=\frac12$, we must have $\theta= \pi = \pi_k$ on $I^3_3$.\end{prop}
\begin{proof}
Let $U = \left[\frac{b}{4}, \frac{3b}{8}\right]\subseteq I_3^3$ and note that $U+U = \left[\frac{b}{2}, \frac{3b}{4}\right]\subseteq I_3^3$. For $x,y\in U$, since $\pi = \pi_k$ on $I_3^3$ we see that
$$\pi(x)+\pi(y) = \frac{1}{b}x+\frac{1}{b}y = \frac{1}{b}\left(x+y\right) = \pi(x+y).$$ Hence $U\times U\subseteq E(\pi)\subseteq E(\theta).$ Using Lemma~\ref{lem:interval-lemma}, $\theta$ is affine over $[\frac{b}{2}, \frac{3b}{4}]$. By assumption, $\theta(\frac{b}{2}) = \pi(\frac{b}{2}) = \frac{1}{2}$. Using this and $(\frac b4, \frac b4)\in E(\pi)\subseteq E(\theta)$, it follows that $\theta(\frac{b}{4}) = \pi(\frac{b}{4}) = \frac14$. Since $\pi$ satisfies the symmetry condition and $E(\pi)\subseteq E(\theta)$, $\theta$ also satisfies the symmetry condition. This implies $\theta(\frac{3b}{4}) = \pi(\frac{3b}{4})=\frac34$. Therefore, by the affine structure of $\theta$ and $\pi$ over $[\frac{b}{2}, \frac{3b}{4}]$, it follows that $\theta=\pi$ on $[\frac{b}{2}, \frac{3b}{4}]$. The symmetric property of $\theta$ and $\pi$ then yields $\theta = \pi$ on $[\frac{b}{4}, \frac{b}{2}]$ and thus on $I^3_3$. \qed\end{proof}

\begin{prop}\label{claim:ik2_induction_general} Let $k \geq 3$ and $j\in \{3, \dots, k\}$. Let $\pi$ be a minimal valid function such that $\pi = \pi_k$ on $I^j_2\cup I^j_4\cup I^j_6$. Then for all locally bounded functions $\theta:\R \to \R_+$ such that $E(\pi) \subseteq E(\theta)$ and $\theta = \pi$ on $I_3^j\cup I_6^j \cup \{1\}$, we must have $\theta= \pi = \pi_k$ on $I^j_2\cup I^j_4$.

%Let $\theta, \pi$ be minimal valid functions so that $E(\pi)\subseteq E(\theta)$. Let $k \geq 3$ and $j\in \{3, \dots, k\}$. Suppose $\theta = \pi$ on $I_3^j\cup I_6^j$ and $\pi = \pi_k$ on $I^j_2\cup I^j_4\cup I_6^j$. Then $\theta = \pi$ on $I^j_2\cup I^j_4$.
\end{prop}
\begin{proof}
Let $U = \left[\frac{3}{2}b\left(\frac{1}{8}\right)^{j-2}, 2b\left(\frac{1}{8}\right)^{j-2}\right]\subseteq I^j_2$ and $V = \left[1-\frac{1}{2}b\left(\frac{1}{8}\right)^{j-2}, 1\right]\subseteq I_6^j\cup\{1\}$. Observe that $U+V \equiv I^j_2$ (modulo 1). Moreover, $x\in U$ and $y\in V$ implies
\begin{equation*}
\begin{array}{rlr}
\pi(x)+\pi(y) = & \left(\frac{4^{2-j}}{1-b}-\left(\frac{1}{1-b}\right)x\right)+\left(\frac{1}{1-b}-\left(\frac{1}{1-b}\right)y\right)&\text{since }\pi = \pi_k\text{ on }I_2^j\cup I_6^j\\
= & \frac{4^{2-j}}{1-b}-\left(\frac{1}{1-b}\right)\left(x+y-1\right)\\
= &\pi(x+y-1) =\pi(x+y)&\text{by periodicity}.\\
\end{array}
\end{equation*}
Thus, $U\times V\subseteq E(\pi)\subseteq E(\theta)$, and by Lemma~\ref{lem:interval-lemma}, $\pi$ and $\theta$ are affine over $I^j_2$ with the same slope as their corresponding slopes over $V$. Since $\theta = \pi = \pi_k$ over $I^j_6$ and $V \subseteq I^j_6\cup\{1\}$, all three functions have the same slope over $I^j_2$. Since $\theta = \pi$ on $I^j_3$ by assumption, it must be that $\theta\left(2b\left(\frac{1}{8}\right)^{j-2}\right) =\pi\left(2b\left(\frac{1}{8}\right)^{j-2}\right) $. Therefore, $\theta = \pi$ on $I^j_2$. Since $\pi$ satisfies the symmetry condition and $E(\pi)\subseteq E(\theta)$, $\theta$ also satisfies the symmetry condition. Using symmetry, we see that $\theta=\pi$ over $I^j_4$. \qed\end{proof}

\begin{prop}\label{claim:Ij1_induction_general} Let $k \geq 4$ and let $j \in \{3, \ldots, k-1\}$. Let $\pi$ be a minimal valid function such that $\pi = \pi_k$ on $I^j_1\setminus\intr(I^{j+1}_1\cup I^{j+1}_2)$ and $I^j_5\setminus \intr(I_4^{j+1}\cup I^{j+1}_5)$. Then for all locally bounded functions $\theta:\R \to \R_+$ such that $E(\pi) \subseteq E(\theta)$ and $\theta = \pi$ on $I^j_2\cup I^j_3 \cup I^j_4 \cup I^j_6$, we must have $\theta= \pi = \pi_k$ on $I^j_1\setminus\intr(I^{j+1}_1\cup I^{j+1}_2)$ and $I^j_5\setminus \intr(I_4^{j+1}\cup I^{j+1}_5)$.

%Let $\theta, \pi$ be minimal valid functions so that $E(\pi)\subseteq E(\theta)$. Let $k \geq 3$ and let $j \in \{3, \ldots, k-1\}$. Suppose $\theta = \pi$ on $I^j_2\cup I^j_3 \cup I^j_4 \cup I^j_6$ and $\pi = \pi_k$ over $I^j_1\setminus\intr(I^{j+1}_1\cup I^{j+1}_2)$ and $I^j_5\setminus \intr(I_4^{j+1}\cup I^{j+1}_5)$. Then $\theta = \pi$ over $I^j_1\setminus\intr(I^{j+1}_1\cup I^{j+1}_2)$ and $I^j_5\setminus \intr(I_4^{j+1}\cup I^{j+1}_5)$.
\end{prop}

\begin{proof}
By minimality, we have that $\pi(0)=\pi_k(0)=0$. Since $E(\pi)\subseteq E(\theta)$, we have that $\pi(0)+\pi(0)=\pi(0)$ implies $\theta(0)+\theta(0)=\theta(0)$. This shows that $\theta(0) = 0=\pi(0) = \pi_k(0)$. Now consider $I^*:=I^j_1\setminus(\intr(I^{j+1}_1\cup I^{j+1}_2) \cup\{0\}) = \left[2b\left(\frac{1}{8}\right)^{j-1}, b\left(\frac{1}{8}\right)^{j-2}\right]$.
Let
$$U=\left[2b\left(\frac{1}{8}\right)^{j-1}, 4b\left(\frac{1}{8}\right)^{j-1}\right] \subseteq I^*.$$
Note that
$$U+U= \left[4b\left(\frac{1}{8}\right)^{j-1}, b\left(\frac{1}{8}\right)^{j-2}\right]\subseteq I^*$$
and $U\cup (U+U) = I^*$. A direct calculation shows that $I^* \subseteq I^m_3$ for all $m\geq j+1$. Thus, by the definition of $\pi_k$, it follows that $\pi_k(x) = \pi_j(x) = (\frac{2^{j-2}-b}{b-b^2})x$ for all $x\in I^*$. Using this and the fact that $\pi = \pi_k$ over $I^*$, we see that
\begin{align*}
\pi(x)+\pi(y)=\pi_k(x) + \pi_k(x) & = \left(\frac{2^{j-2}-b}{b-b^2}\right)x+\left(\frac{2^{j-2}-b}{b-b^2}\right)y\\
=&\left(\frac{2^{j-2}-b}{b-b^2}\right)(x+y)=\pi_k(x+y) = \pi(x+y)
\end{align*}
for $x,y\in U$, and so $U\times U\subseteq E(\pi)\subseteq E(\theta)$. By Lemma~\ref{lem:interval-lemma}, $\theta$ is affine over $U+U$ and $U$ with the same slope, and thus affine over $I^*$. Similarly, $\pi$ is affine over $I^*$.

Since $\theta=\pi$ on $I^j_2$, we have $\theta\left(b\left(\frac{1}{8}\right)^{j-2}\right)=\pi\left(b\left(\frac{1}{8}\right)^{j-2}\right)$. Also, since $2b\left(\frac{1}{8}\right)^{j-1}, 4b\left(\frac{1}{8}\right)^{j-1}\in U$ and $U\times U\subseteq E(\pi)\subseteq E(\theta)$, we see that

\begin{align*}
4\theta\left(2b\left(\frac{1}{8}\right)^{j-1}\right) & = 2\theta\left(2b\left(\frac{1}{8}\right)^{j-1}+2b\left(\frac{1}{8}\right)^{j-1}\right) \\
& = 2\theta\left(4b\left(\frac{1}{8}\right)^{j-1}\right) = \theta\left(4b\left(\frac{1}{8}\right)^{j-1}+4b\left(\frac{1}{8}\right)^{j-1}\right)\\
& = \theta\left(b\left(\frac{1}{8}\right)^{j-2}\right)\\
& = \pi\left(b\left(\frac{1}{8}\right)^{j-2}\right)\qquad \text{since }\theta = \pi\text{ on }I_2^j&&\\
& = 4\pi\left(2b\left(\frac{1}{8}\right)^{j-1}\right).
\end{align*}
Thus, $\theta\left(2b\left(\frac{1}{8}\right)^{j-1}\right) = \pi\left(2b\left(\frac{1}{8}\right)^{j-1}\right)$, and so $\theta = \pi$ on the endpoints of $I^*$. Since both functions are affine on $I^*$, it follows that $\theta = \pi$ on $I^*$. Since $\pi$ satisfies the symmetry condition and $E(\pi)\subseteq E(\theta)$, $\theta$ also satisfies the symmetry condition. Symmetry of $\theta$ and $\pi$ yields that $\theta = \pi$ over $I^j_5\setminus\intr(I^{j+1}_4\cup I^{j+1}_5)$.\qed
\end{proof}

\begin{lemma}\label{lem:main-lem} Let $k \geq 3$ and $j\in \{3, \dots, k\}$. Let $\pi$ be a minimal valid function such that $\pi = \pi_k$ on $I^j_3 \cup I^j_6$. Then for all locally bounded functions $\theta:\R\to \R_+$ such that $E(\pi) \subseteq E(\theta)$ satisfying $\theta(0)=0,\theta(b) = 1$, we must have $\theta= \pi = \pi_k$ on $I^j_3 \cup I^j_6$.

%Let $\theta, \pi$ be minimal valid functions so that $E(\pi)\subseteq E(\theta)$. Let $k \geq 3$ and $j \in \{3, \ldots, k\}$ so that $\pi = \pi_k$ on $I^j_3 \cup I^j_6$. Then we must have $\theta= \pi$ on $I^j_3 \cup I^j_6$.
\end{lemma}
\begin{proof}

By Proposition~\ref{claim:Ik6_general}, we obtain $\theta= \pi$ on $I^k_6 = I^j_6= I^3_6$. We prove $\theta= \pi$ on $I^j_3$ by induction on $j$. For $j = 3$, the result follows from Proposition~\ref{claim:i33_general} (observe that $E(\pi) \subseteq E(\theta)$ implies that $\theta$ is symmetric and therefore $\theta (\frac{b}{2}) = \frac12$). We assume the result holds for some $j$ such that $3 \leq j \leq k-1$ and show that it holds for $j+1$. Observe that this assumption implies that $k\geq 4$. Note that $I^{j+1}_3 \cup\{0,b\} = (I^{j}_1\setminus\intr(I^{j+1}_1\cup I^{j+1}_2)) \cup I^j_2 \cup I^j_3 \cup I^j_4 \cup (I^{j}_5\setminus\intr(I^{j+1}_4\cup I^{j+1}_5))$. By the induction hypothesis, $\theta = \pi$ on $I^j_3$. Since $\pi(0)=\pi(1)=0$, it follows that $\pi(0) = \pi(1) = \pi(0)+\pi(1)$. Thus, since $E(\pi) \subseteq E(\theta)$, we have $ 0 = \theta(0) = \theta(0)+\theta(1) = \theta(1)$. By Proposition~\ref{claim:ik2_induction_general}, $\theta = \pi$ on $I^j_2\cup I^j_4$. Using this, Proposition~\ref{claim:Ij1_induction_general} implies that $\theta = \pi$ on $(I^{j}_1\setminus\intr(I^{j+1}_1\cup I^{j+1}_2))\cup (I^{j}_5\setminus\intr(I^{j+1}_4\cup I^{j+1}_5))$.
%Then from Propositions~\ref{claim:ik2_induction_general} and~\ref{claim:Ij1_induction_general}, it follows that $\theta = \pi$ on the rest of $I_3^{j+1}$.
\qed
\end{proof}

\begin{proof}[of Proposition~\ref{prop:pi_k_facet}]
If $k=2$, then the fact that $\pi_k$ is a facet follows by Theorem~\ref{thm:2-slope}. Consider the setting $k\geq 3$.

Let $\theta:\R\to \R_+$ be a minimal valid function for $I_b$ such that $E(\pi_k)\subseteq E(\theta)$. Since $\theta$ is minimal, a consequence of Theorem~\ref{thm:minimalinteger} is that $\theta$ is locally bounded. Using $\pi = \pi_k$ in  Lemma~\ref{lem:main-lem}, it follows that $\theta = \pi_k$ on $I^k_3\cup I^k_6$. From Proposition~\ref{claim:ik2_induction_general} and again setting $\pi = \pi_k$, we obtain that $\theta = \pi_k$ on $I_2^k\cup I_4^k$. It is left to show that $\theta = \pi_k$ on $I^k_1$ and $I^k_5$.

Let $U =  \left[0,\frac{b}{2}\left(\frac{1}{8}\right)^{k-2}\right]$ and observe that $U+U  = \left[0, b\left(\frac{1}{8}\right)^{k-2}\right]=I^k_1$. It follows from the definition of $\pi_k$ that $\pi_k(x)+\pi_k(y)=\pi_k(x+y)$ for all $x,y, x+y\in I^k_1$, so $U\times U \subseteq E(\pi_k)\subseteq E(\theta)$. Since $\theta$ and $\pi_k$ are minimal, $\theta(0) = \pi_k(0)=0$. Also, since $\theta = \pi_k$ on $I_2^k$, $\theta\left(b\left(\frac{1}{8}\right)^{k-2}\right) = \pi_k\left(b\left(\frac{1}{8}\right)^{k-2}\right)$. Thus $\theta = \pi_k$ on the endpoints of $I^k_1$. Moreover, Lemma~\ref{lem:interval-lemma} implies that $\theta$ is affine over $I_1^k$. Since $\pi_k$ is also affine over $I_1^k$ and $\theta = \pi_k$ at the endpoints, we have $\theta = \pi_k$ on $I_1^k$. The fact that $\theta = \pi_k$ on $I^k_5$ follows by symmetry. Therefore, $\theta = \pi_k$ everywhere. By Theorem~\ref{thm:facet}, $\pi_k$ is a facet.\qed
\end{proof}

\begin{proof}[of Theorem~\ref{thm:kslopes}]
By Proposition~\ref{prop:slopes}, the function $\pi_k$ is piecewise linear and has $k$ slopes. Every valid function that is a facet is also extreme. By Proposition~\ref{prop:pi_k_facet}, $\pi_k$ is a facet (and therefore extreme). Thus, $\pi_k$ proves the result. \qed
\end{proof}

%\begin{theorem} There exists a continuous extreme function $\pi_{\infty}$ with an infinite number of slopes.
%\end{theorem}
\section{Proof of Theorem~\ref{cor:infinite_slopes}}\label{sec:proof-infinite-slopes}

%For this proof, we use the following lemma (see Lemma 2.11 (ii) in~\cite{bhk-survey}).
%
%\begin{lemma}\label{lem:minimal-cons} Let $\theta$ be a minimal valid function and $\theta_1,\theta_2$ be valid functions such that $\theta = \frac{\theta_1+\theta_2}{2}$. Then $\theta_1, \theta_2$ are minimal and for all $x, y \in \R$, $\theta(x + y) = \theta(x) + \theta(y)$ implies $\theta_i(x+y) = \theta_i(x) + \theta_i(y)$ for both $i=1,2$.
%\end{lemma}

\begin{proof}[of Theorem~\ref{cor:infinite_slopes}]

For $x\in [0,1)\setminus \{0,b\}$, let $N_x$ be the natural number guaranteed by Proposition~\ref{prop:properties} $(iii)$, that is $x \in I^{N_x}_3\cup I^{N_x}_6$ and $\pi_{k}(x) = \pi_{N_x}(x)$ whenever $k\geq N_x$. Define the function $\pi_{\infty}(x):\R\to [0,1]$ by 

\begin{equation*}
\pi_{\infty}(x) = \begin{cases}
0,&x=0\\
1,&x=b\\
\pi_{N_x}(x),&x\in [0,1)\setminus \{0,b\}\\
\pi_{\infty}(x-j),&x\in [j,j+1),~j\in \Z\setminus\{0\}~.
\end{cases}
\end{equation*}

We claim that the sequence $\{\pi_i\}_{i=2}^\infty$ converges uniformly to $\pi_{\infty}$. To this end, let $\epsilon>0$. Choose a large enough $N\in \N$ such that $\frac{2^{4-3N}(2^N-4b)}{1-b} < \epsilon$. Let $x\in \R$ and $k\geq N$. We consider cases on $x$. 

\textbf{Case 1:} Assume $x\in I^k_1\cup I^k_2$. By Proposition~\ref{prop:properties} $(i)$, $I^m_1\cup I^m_2\supseteq I^{m+1}_1\cup I^{m+1}_2$ for all $m\in \N$. For $m\geq k$, the maximum value of $\pi_m$ on $I_1^m\cup I_2^m$ is $\frac{2^{4-3m}(2^m-4b)}{1-b}\leq \frac{2^{4-3k}(2^k-4b)}{1-b}\leq \frac{2^{4-3N}(2^N-4b)}{1-b}$. It follows that $0\leq \pi_{\infty}(x)\leq \frac{2^{4-3N}(2^N-4b)}{1-b}<\epsilon$. Similarly, $0\leq \pi_k(x)\leq \epsilon$. Thus, $|\pi_k(x)-\pi_{\infty}(x)|<\epsilon$. 

\textbf{Case 2:} Assume $x\in I^k_4\cup I^k_5$. Then $|\pi_{\infty}(x)-\pi_k(x)|<\epsilon$ follows by the symmetry of each function in the sequence $\{\pi_i\}_{i=2}^\infty$ along with Case 1.

\textbf{Case 3:} Assume $x\in I^k_3\cup I^k_6$. Note that $I^k_3\cup I^k_6\subseteq  I^m_3\cup I^m_6$ for every $m\in \N$ such that $m\geq k$. Therefore, by definition of each $\pi_m$ for $m\geq k$, we see that $\pi_m(x) = \pi_k(x)$ for all $m\geq k$. Hence $\pi_{\infty}(x) = \pi_k(x)$. 

\textbf{Case 4:}  Assume $x\in \{0,b\}$. Then $\pi_{\infty}(x) = \pi_k(x)$ follows directly by definition of $\pi_{\infty}(0)$ and $\pi_{\infty}(b)$. 

\textbf{Case 5:} Assume that $x\in [j, j+1)$ for $j\in \Z\setminus \{0\}$. Then using the periodicity of each function in $\{\pi_i\}_{i=2}^\infty$ and noting $N_x = N_{x-j}$, we obtain $|\pi_{\infty}(x) - \pi_k(x)| = |\pi_{\infty}(x-j) - \pi_k(x-j)|<\epsilon$ by using Cases 1-4.

Since each $\pi_k$ is minimal, by a standard limit argument, $\pi_\infty$ is minimal (Proposition 4 in \cite{dey1}, Proposition 6.1 in \cite{basu2016light2}). Also, since $\pi_{\infty}$ is the uniform limit of continuous functions, it too is continuous.  
%
%Using Proposition \ref{prop:properties} $(v)$, $\pi_{\infty}$ is continuous over $(0,b)$ and $(b,1)$. For $x=0$ or $x=b$, note that, by definition of $\pi_k$, the maximum value of $\pi_k$ on $I_1^k\cup I_2^k$ is $\frac{2^{4-3k}(2^k-4b)}{1-b}$, which tends to 0 as $k\to\infty$. By symmetry, the smallest value of $\pi_k$ on the interval $I_4^k\cup I_5^k$ tends to $1$ as $k\to\infty$. Hence, the convergence $\pi_k\to \pi_\infty$ is actually uniform. Therefore $\pi_{\infty}$ is continuous everywhere.

We next show that $\pi_{\infty}$ is a facet. Let $\theta$ be any minimal function such that $E(\pi_\infty) \subseteq E(\theta)$. If $x=0$ or $x=b$, then $\pi_{\infty}(x) = \theta(x)$ by the minimality of $\pi_{\infty}$ and $\theta$. So assume that $x\not\in\{0, b\}$. Recall that $x \in I^{N_x}_3\cup I^{N_x}_6$. Observe that $\pi_\infty = \pi_{N_x}$ on $I^{N_x}_3\cup I^{N_x}_6$. Hence, by applying Lemma~\ref{lem:main-lem} with $k = j = N_x$ and $\pi = \pi_\infty$, we obtain that $\theta(x) = \pi_\infty(x)$. Therefore, $\theta = \pi_\infty$ everywhere. By Theorem~\ref{thm:facet}, $\pi_\infty$ is a facet.

We finally verify that $\pi_\infty$ has infinitely many slopes. Note that for any $k \geq 3$, $\pi_\infty = \pi_k$ on $I^k_3\cup I^k_6$ and, by Proposition~\ref{prop:slopes}, $\pi_k$ has $k-1$ different slopes on $I^k_3\cup I^k_6$.\qed
\end{proof}

\section{Facets for higher dimensional group relaxations}\label{section:seq_merge}

One can ask if it is possible to find extreme functions with arbitrary number of slopes for the higher-dimensional infinite group relaxation. For $b\in \R\setminus \Z$, a trivial way to generalize to higher dimensions is to simply define $\pi^n_k : \R^n \to \R_+$ as $\pi^n_k(x_1, x_2, \ldots, x_n) = \pi_k(x_1)$ and $\pi^n_\infty:\R^n \to \R_+$ as $\pi^n_\infty(x_1, x_2, \ldots, x_n) = \pi_\infty(x_1)$. By Theorem 19.35 in~\cite{Richard-Dey-2010:50-year-survey}, the functions $\pi^n_k$ and $\pi^n_\infty$ are extreme for $I_{\tilde{b}}$ for $\tilde{b}\in \{b\}\times \{0\}^{n-1}$. However, one can ask whether there are more ``non-trivial'' examples. In particular, one can ask whether there exist {\em genuinely $n$-dimensional} extreme functions with arbitrary number of slopes for all $n\geq 1$. A function $\theta:\R^n \to \R$ is genuinely $n$-dimensional if there does not exist a linear map $T:\R^n \to \R^{n-1}$ and a function $\theta':\R^{n-1} \to \R$ such that $\theta = \theta' \circ T$. The construction of such a ``non-trivial'' facet is the main result in this section. We use the notation $\chf_m$ to denote the vector of all ones in $\R^m$ and $b\chf_m$ to denote the vector in $\R^m$ such that every component is equal to $b$.

\begin{theorem}\label{thm:seq_merge}
Let $n,k\in \N$. For any $b\in \R\setminus \Z$, there exists a function $\Pi^n_k:\R^n\to \R_+$ such that $\Pi^n_k$ has at least $k$ slopes, is genuinely $n$ dimensional, and is a facet (and thus extreme) for the $n$-dimensional infinite group relaxation $I_{b\mathbf{1}_n}$.
\end{theorem}

We provide a constructive argument for the proof of Theorem~\ref{thm:seq_merge} using the {\em sequential-merge} operation developed by Dey and Richard~\cite{dey2}. In particular, we employ Theorem 5 in~\cite{dey2}, the assumptions of which will be proved throughout this section. The proof of Theorem~\ref{thm:seq_merge} is the collection of these results and is presented at the end of the section. We begin with some definitions relating to sequential-merge.
\medskip

\noindent{\textbf{Notation:}} Let $m\in \N$ and $x,y\in \R^m$. For $i\in \{1, \dots, m\}$, we use the notation $x_i$ to denote the $i$-th component of $x$. We define the vectors $\lfloor x \rfloor := (\lfloor x_1\rfloor, \lfloor x_2\rfloor, ...,\lfloor x_n\rfloor)\in \R^m$ and $x_{-1}:=(x_2, \ldots, x_m)\in \R^{m-1}$. We use the notation $x\leq y$ to indicate that $x_i\leq y_i$ for each $i\in \{1, \dots, m\}$.  \hspace*{\stretch{1}} $\diamond$

\smallskip

%For $n\geq 2$, a function $\theta:\R^n\to \R$ is \emph{piecewise linear} if $\R^n$ can be divided into polytopes such that $\theta$ is affine over each polytope, see~\cite{dey2}. %{\color{red}JOE: Does this need to be more formal? They use a similar definition in~\cite{dey2}}

Let $b\in [0,1)^n \setminus \{0\}$. The \textit{lifting-space representation} of any function $\theta:\R^n\to \R$ is $[\theta]_b:\R^n \to \R$ defined by
\begin{equation*}
[\theta]_b(x) := \sum_{i=1}^nx_i-\sum_{i=1}^nb_i\theta\left(x-\lfloor x\rfloor\right).
\end{equation*}
\begin{remark}\label{rem:diff-lifting} The definition for lifting-space representation in~\cite{dey2} is given only for valid functions which in that context are periodic modulo $\Z^n$. Note that if $\theta$ is periodic, then $[\theta]_b(x) = \sum_{i=1}^nx_i-\sum_{i=1}^nb_i\theta\left(x\right)$.\hspace*{\stretch{1}} $\diamond$
\end{remark}
The \textit{group-space representation} of any function $\psi:\R^n\to \R$ is $[\psi]^{-1}_b:\R^n\to \R$ defined by
\begin{equation*}
[\psi]_b^{-1}(x) := \frac{\sum_{i=1}^n x_i - \psi(x)}{\sum_{i=1}^n b_i}.
\end{equation*}
A function $\psi:\R^n\to \R$  is called {\em superadditive} if $-\psi$ is subadditive. $\psi$ is called {\em pseudo-periodic} if $\psi(x+e^i) = \psi(x)+1$ for all standard unit vectors $e^i\in \R^n$ and $x\in \R^n$. %Note that for pseudo-periodic functions $\theta$, The next result follows from Proposition 3 in~\cite{dey2}.

We collect some useful facts above the above definitions below.

\begin{prop}\label{obs:represent} Let $b\in [0,1)^n\setminus \{0\}$. 
\begin{enumerate}
\item[(i)] If $\psi:\R^n \to \R$ is pseudo-periodic, then $[\psi]^{-1}_b$ is periodic modulo $\Z^n$.
%\item[(ii)] $\theta:\R^n \to \R$ is superadditive if and only if $[\theta]^{-1}_b$ is subadditive.
\item[(ii)] If $\pi$ is a minimal valid function for $I_b$, then $[\pi]_b$ is superadditive and pseudo-periodic.
\item[(iii)] If $\psi$ is pseudo-periodic then $[[\psi]^{-1}_b]_b = \psi$.%, i.e., $[\cdot]_b$ is a left-inverse for $[\cdot]^{-1}_b$ (as function composition).
\end{enumerate}

\end{prop}

\begin{proof} If $\psi:\R^n \to \R$ is pseudo-periodic, then observe that for any $x\in \R^n$ and unit vector $e^i$, $\psi(x) = \psi((x-e^i)+e^i) = \psi(x-e^i) + 1$, i.e., $\psi(x-e^i) = \psi(x) - 1$. By iterating, we observe that $\psi(x + z) = \psi(x) + \sum z_i$ for all $x \in \R^n$ and $z \in \Z^n$. Therefore, $[\psi]_b^{-1}(x + z) = \frac{\sum_{i=1}^n x_i + \sum_{i=1}^n z_i- \psi(x) - \sum_{i=1}^nz_i }{\sum_{i=1}^n b_i} = \frac{\sum_{i=1}^n x_i - \psi(x)}{\sum_{i=1}^n b_i} = [\psi]_b^{-1}(x)$. This establishes $(i)$.
\medskip

%(ii) follows from the definition $[\psi]_b^{-1}(x) = \frac{\sum_{i=1}^n x_i}{\sum_{i=1}^n b_i} - \frac{1}{\sum_{i=1}^n b_i}\psi(x)$ and the facts that $\frac{\sum_{i=1}^n x_i}{\sum_{i=1}^n b_i}$ is a linear function of $x$, and $\sum_{i=1}^n b_i\geq 0$.
%\medskip

$(ii)$ follows from Proposition 3 in~\cite{dey2}.
\medskip

To establish $(iii)$, first observe that by $(i)$ above $[\psi]^{-1}_b$ is periodic modulo $\Z^n$. Therefore, by Remark~\ref{rem:diff-lifting}, $[[\psi]^{-1}_b]_b(x) = \sum_{i=1}^nx_i-\sum_{i=1}^nb_i[\psi]^{-1}_b\left(x\right) = \sum_{i=1}^nx_i-\sum_{i=1}^nb_i\frac{\sum_{j=1}^n x_j - \psi(x)}{\sum_{j=1}^n b_j} = \psi(x)$. \qed\end{proof}

\begin{remark}\label{rem:diff-group} The definition for group-space representation in~\cite{dey2} is given only for superadditive and pseudo-periodic $\psi$, and the domain of $[\psi]_b^{-1}(x)$ is defined to be $[0,1)^n$. We give the definition for more general functions and allow the domain of $[\psi]_b^{-1}(x)$ to be $\R^n$ to make calculations easier. By Proposition~\ref{obs:represent} $(i)$, restricted to pseudo-periodic  functions, our definition of $[\psi]_b^{-1}(x)$ is simply a periodization of the definition from~\cite{dey2}. \hspace*{\stretch{1}} $\diamond$%\footnote{In fact, if $\psi$ is superadditive, then $[\psi]_b^{-1}(x)$ is subadditive, and vice versa, but we will not need this in this paper.}
\end{remark}

Let $b_1\in (0,1)$ and $b_2\in [0,1)^m\setminus \{0\}$. If $f:\R \to \R_+$ and $g:\R^m\to \R_+$ are minimal valid functions for $I_{b_1}$ and $I_{b_2}$, respectively, then the \textit{sequential merge} $f\diamond g: \R\times \R^{m}\to \R$ is defined as
\begin{equation*}
f\diamond g := \left[\psi \right]^{-1}_{(b_1, b_2)}
%\left[[\pi]_{b_1}(x_1-\lfloor x_1\rfloor+[\psi]_{b_2}(x_2-\lfloor x_2\rfloor))\right]^{-1}_{(b_1, b_2)}(x_1-\lfloor x_1\rfloor, x_2-\lfloor x_2\rfloor ).
\end{equation*}
where $\psi: \R \times \R^m \to \R$ is the function $\psi(x_1, x_2) =  [f]_{b_1}(x_1 + [g]_{b_2}(x_2))$.

\begin{remark}\label{rem:diff-seq-merge} It is not true in general that if $f, g$ are minimal, then $f\diamond g$ is minimal. However, when $f, g$ are minimal valid functions, $f\diamond g$ is periodic modulo $\Z^n$. Indeed, since $f\diamond g := \left[\psi \right]^{-1}_{(b_1, b_2)}$, by Proposition~\ref{obs:represent} $(i)$ it suffices to check that $\psi(x_1, x_2)$ defined above is pseudo-periodic when $f, g$ are minimal. By Proposition~\ref{obs:represent} $(ii)$, $[f]_{b_1}, [g]_{b_2}$ are both pseudo-periodic. Therefore, $\psi(x_1 + 1, x_2) = [f]_{b_1}(x_1 + 1 + [g]_{b_2}(x_2)) = [f]_{b_1}(x_1 + [g]_{b_2}(x_2)) + 1 = \psi(x_1,x_2)+1$, using pseudo-periodicity of $[f]_{b_1}$. On the other hand, for any unit vector $e^i \in \R^m$, we have $\psi(x_1, x_2 + e^i) = [f]_{b_1}(x_1 + [g]_{b_2}(x_2 + e^i)) = [f]_{b_1}(x_1+ [g]_{b_2}(x_2) + 1) = [f]_{b_1}(x_1 + [g]_{b_2}(x_2)) + 1=\psi(x_1,x_2)+1$, where the second equality uses pseudo-periodicity of $[g]_{b_2}$ and the last equality uses pseudo-periodicity of $[f]_{b_1}$. %Since $f\diamond g := \left[\psi \right]^{-1}_{(b_1, b_2)}$, by Proposition~\ref{obs:represent}$(i)$, $f\diamond g$ is periodic.

In Dey and Richard's original definition from~\cite{dey2}, the domain of $f\diamond g$ is defined as $[0,1) \times [0,1)^m$, and restricted to this domain, our definition is exactly the same as theirs. Thus, our definition over $\R \times \R^m$ is simply a periodization of Dey and Richard's definition for $f\diamond g$, when $f, g$ are minimal functions. Since we will only apply the sequential merge operation on minimal valid functions, there is no discrepancy between the definition in~\cite{dey2} and our definition. \hspace*{\stretch{1}} $\diamond$
\end{remark}

For the remainder of this section, we consider $b\in [1/2, 1)$. Although the specific construction of $\pi_k$ provided in Section~\ref{sect-constr} uses $b\in (0,1/2]$, creating $\pi_k$ for $b\in [1/2, 1)$ can be done by defining $\pi_k(x) := \tilde{\pi}_k(1-x)$ for $x\in [0,1]$ (and then enforcing periodicity by $\Z$), where $\tilde{\pi}_k$ is the function for $I_{1-b}$ constructed in Section~\ref{sect-constr} (see also Theorem~\ref{thm:reflection}). 

Let $\phi$ denote the GMI function for $I_b$ (defined in~\eqref{Gom-funct}). For $n\in \N$, $n\geq 2$, let $\Pi^n_k:\R^n\to \R$ be defined by
\begin{equation*}
\Pi^n_k(x_1, \dots, x_n) := \pi_k\diamond \left(\phi\diamond\left(\phi\diamond\left(...\diamond\phi\right)...\right)\right)(x_1, \dots, x_n),
\end{equation*}
where the sequential merge contains one copy of $\pi_k$ and $n-1$ copies of $\phi$. For $m \in \N$ and $m \geq 1$, let $\Phi_m$ denote $\phi\diamond\left(\phi\diamond\left(...\diamond\phi\right)...\right)$, where there are $m$ copies of $\phi$ in the sequential merge. %One can show using induction on $m$ that $\Phi_m(x)=0$ if and only if $x\in \Z^m$.

%A nice formula for the sequential merge procedure was stated in Proposition 5 of~\cite{dey2} and is provided below. %We prove it here for completeness.
%
%\begin{prop}\label{obs:formula} For any $v = (v_1, \ldots, v_n) \in \R^n$, \begin{equation*} \Pi_k(v_1, v_2, \dots, v_n) = \frac{ (n-1)\Phi_{n-1}(v_2, \dots, v_n)+\pi_k\bigg(\sum_{i=1}^nv_i-b\Phi_{n-1}(v_2, \dots, v_n)\bigg)}{n}
%\end{equation*}
%\end{prop}
%

We require a couple of definitions (also taken from~\cite{dey2}) before we proceed with the proof of Theorem~\ref{thm:seq_merge}. 

\begin{enumerate}
\item For $m\in \N$, a function $\theta: \R^m \to \R$ is {\em nondecreasing} if for all $x, y \in \R^m$, $x \leq y$ implies $\theta(x) \leq \theta(y)$.

\item For $m\in \N$ and a valid function $\pi :\R^m \to \R_+$, the set $E(\pi)$ defined in~\eqref{eq:equality} is said to be {\em unique up to scaling} if for any continuous nonnegative function $\theta: \R^m \to \R_+$ satisfying $E(\pi)\subseteq E(\theta)$, $\theta$ is a scaling of $\pi$, i.e,  $\theta = \alpha \pi$ where $\alpha \in \R$.
\end{enumerate}

\begin{remark}\label{remark:GMI_upts}
Dey and Richard note that every extreme function for $I_b$ is unique up to scaling, see the top of page 6 in~\cite{dey2}. In particular, the GMI function $\phi = \pi_2$ is unique up to scaling. \hspace*{\stretch{1}} $\diamond$
\end{remark}

As mentioned in Remark~\ref{rem:diff-seq-merge}, $f\diamond g$ is not necessarily minimal even if $f,g$ are both minimal. The following proposition gives conditions under which $f\diamond g$ is indeed minimal and will be useful in what follows.
\begin{prop}\label{prop:preserve_minimal}\cite[Proposition 7]{dey2}
Let $m\in \N$, $b^1\in (0,1)$, and $b^2\in [0,1)^m\setminus \{0\}$. Let $f:\R \to \R$ be a minimal function for $I_{b^1}$ and $g:\R^m\to \R$ be a minimal function for $I_{b^2}$ such that $[f]_{b^1}$ is nondecreasing. Then $f\diamond g$ is minimal for $I_{(b^1, b^2)}$. 
\end{prop}

To deduce that $f \diamond g$ is a {\em facet}, one needs additional assumptions. We now state the main theorem that guarantees facetness from the sequential-merge operation, due to Dey and Richard~\cite{dey2}.

\begin{theorem}\label{thm:dey-richard-seq-merge}\cite[Theorem 5]{dey2}
Let $m\in \N$, $b^1\in (0,1)$, and $b^2\in [0,1)^m\setminus \{0\}$. Let $f:\R \to \R$ be a minimal function for $I_{b^1}$ and $g:\R^m\to \R$ be a minimal function for $I_{b^2}$ such that the following hold:
\begin{enumerate}
\item $f$ and $g$ are piecewise linear, continuous functions,
\item $[f]_{b^1}$ and $[g]_{b^2}$ are both nondecreasing,%{\red $[\cdot]_{b}$ is defined only for minimal functions.}
\item $E(f)$ and $E(g)$ are unique up to scaling, and
\item $f$ and $g$ are facets for their respective infinite group relaxations.
\end{enumerate}
Then $f \diamond g$ is a facet for $I_{(b^1,b^2)}$.\footnote{The definition of facet used in~\cite{dey2} is slightly different from our definition, and corresponds to what the authors in~\cite{basu2016light} refer to as {\em weak facet}. However, the proof in~\cite{dey2} works for the definition of facet used in this current manuscript. Moreover, we insist on $f, g$ to be minimal valid functions, whereas Dey and Richard consider valid functions that are periodic modulo $\Z^n$, which is a slightly weaker hypothesis than minimality.}
\end{theorem}

We will prove that $\Pi^n_k$ is a facet of $I_{b\chf_n}$ by showing that $\pi_k$ and $\Phi_{n-1}$ satisfy all the conditions of Theorem~\ref{thm:dey-richard-seq-merge}. We divide these into subsections to help organize our arguments.

%\begin{proof} %We first show that $f\diamond g$ is periodic, and therefore the definition of $f\diamond g$ coincides with the definition in~\cite{dey2}. Since $f\diamond g := \left[\psi \right]^{-1}_{(b_1, b_2)}$ for $\psi(x_1, x_2) =  [f]_{b_1}(x_1 + [g]_{b_2}(x_2))$, it suffices to show that $\psi$ is pseudo-periodic by Proposition~\ref{obs:represent}$(i)$. By Proposition~\ref{obs:represent}(iii), $[f]_{b_1}, [g]_{b_2}$ are both pseudo-periodic. Therefore, $\psi(x_1 + 1, x_2) = [f]_{b_1}(x_1 + 1 + [g]_{b_2}(x_2)) = [f]_{b_1}(x_1 + [g]_{b_2}(x_2)) + 1 = \psi(x_1,x_2)+1$, using pseudo-periodicity of $[f]_{b_1}$. On the other hand, for any unit vector $e^i \in \R^m$, we have $\psi(x_1, x_2 + e^i) = [f]_{b_1}(x_1 + [g]_{b_2}(x_2 + e^i)) = [f]_{b_1}(x_1+ [g]_{b_2}(x_2) + 1) = [f]_{b_1}(x_1 + [g]_{b_2}(x_2)) + 1=\psi(x_1,x_2)+1$, where the second equality uses pseudo-periodicity of $[g]_{b_2}$ and the last equality uses pseudo-periodicity of $[f]_{b_1}$. 
%
%Thus, $f\diamond g$ coincides with the definition in~\cite{dey2} and the result follows from Proposition 7 in~\cite{dey2}.
%\end{proof}
%

\subsection{Minimality of $\pi_k, \Phi_{n-1}$}\label{sec:minimal-bb}

The minimality of $\pi_k$, $k\geq 2$ was established in Section~\ref{sect-minimal}; we concentrate on $\Phi_{n-1}$.
\begin{prop}\label{prop:pik_nondecreasing} Let $b\in [1/2, 1)$. The function $[\pi_k]_b$ is nondecreasing for every $k\geq 2$. 
\end{prop}

\begin{proof}
Let $x,y\in \R$ such that $x<y$. %{\red are you implicitly assuming $x, y\in [0,1)$? we should point out that this can be done}. 
By the definition of the lifting-space representation of $\pi_k$ and Remark~\ref{rem:diff-lifting}, we see that
$$[\pi_k]_b(y)-[\pi_k]_b(x) = (y-b\pi_k(y))-(x-b\pi_k(x)) = (y-x) -b(\pi_k(y)-\pi_k(x)).$$
If $[\pi_k]_b(y)< [\pi_k]_b(x)$, then $\frac{1}{b} < \frac{\pi_k(y)-\pi_k(x)}{y-x}$. However, this contradicts that the largest slope (and the only positive slope) in $\pi_k$ is $\frac{1}{b}$ (this crucially uses the fact that we are using $\pi_k$ with $b \in [1/2, 1)$). Thus, $[\pi_k]_b$ is nondecreasing.
\qed
\end{proof}

%

%Let $\Phi_m$ denote $\phi\diamond\left(\phi\diamond\left(...\diamond\phi\right)...\right)$, where there are $m$ copies of $\phi$ in the sequential-merge. In order to use the sequential-merge operator in the definition of $\Phi_m$, we can use Proposition~\ref{prop:preserve_minimal} and show that $\phi$ and $\Phi_{m-1}$ are minimal functions for $R_b(\R, \Z)$ and $R_{b\chf_{m-1}}(\R^{m-1}, \Z^{m-1})$, respectively, and $[\phi]_b$ is nondecreasing. We have shown that $\phi = \pi_k$ is minimal, and Proposition~\ref{prop:pik_nondecreasing} shows that $[\pi]_b$ is nondecreasing. The following result proves the final piece.

\begin{prop}\label{prop:Phi_m_minimal}
Let $b\in [1/2, 1)$. Then $\Phi_m$ is minimal for $I_{b\chf_m}$ for every $m\in \N$.
\end{prop}
\begin{proof}
We proceed by induction on $m$. If $m=1$, then $\Phi_m = \phi = \pi_2$ and the result follows from Theorem~\ref{thm:minimalinteger}. So assume that $\Phi_m$ is minimal for $I_{b\chf_m}$ for $m\in \N$, and consider $\Phi_{m+1}$. Note that $\Phi_{m+1} = \phi\diamond \Phi_m$. From Proposition~\ref{prop:pik_nondecreasing}, $[\phi]_b = [\pi_2]_b$ is nondecreasing. Since $\phi$ and $\Phi_m$ are minimal by the induction hypothesis, $\Phi_{m+1}$ is minimal for $I_{b\chf_{m+1}}$ by Proposition~\ref{prop:preserve_minimal}. 
\qed
\end{proof}
%Proposition~\ref{prop:Phi_m_minimal} implies that the function $[\Phi_m]_{b\chf_m}$ is well-defined for every $m\in \N$.  

%%
%We are now ready to define our proposed functions to prove Theorem~\ref{thm:seq_merge}. Let $\Pi_k:\R^n\to \R$ be defined by
%\begin{equation*}
%\Pi_k(x_1, \dots, x_n) := \pi_k\diamond \left(\phi\diamond\left(\phi\diamond\left(...\diamond\phi\right)...\right)\right)(x_1, \dots, x_n),
%\end{equation*}
%where the sequential-merge contains one copy of $\pi_k$ and $n-1$ copies of $\phi$. By Propositions\ref{prop:pik_nondecreasing} and~\ref{prop:preserve_minimal}, the use of the sequential-merge operator is valid in the definition of $\Pi_k$. 
%
%%

\subsection{$\pi_k$ and $\Phi_{n-1}$ are piecewise linear and continuous}\label{sec:pwl-cont}

$\pi_k$ is piecewise linear and continuous by Propositions~\ref{prop:nonnegativity} and \ref{prop:slopes}. We analyze $\Phi_{n-1}$. A nice formula for the sequential-merge procedure was stated in Proposition 5 of~\cite{dey2} and is applied to $\Phi_m$ and $\Pi^n_k$ below. 

\begin{prop}\label{obs:formula} Let $b\in [1/2, 1)$. For $m\in \N$ with $m\geq 2$ and $x\in \R^m$,
\begin{equation*} \Phi_m(x) = \frac{ (m-1)\Phi_{m-1}(x_{-1})+\phi\bigg(\sum_{i=1}^mx_i-(m-1)b\Phi_{m-1}(x_{-1})\bigg)}{m}.
\end{equation*}

For $k\in \N$ with $k\geq 2$ and $x \in \R^n$, \begin{equation*} \Pi^n_k(x) = \frac{ (n-1)\Phi_{n-1}(x_{-1})+\pi_k\bigg(\sum_{i=1}^nx_i-(n-1)b\Phi_{n-1}(x_{-1})\bigg)}{n}.
\end{equation*}

\end{prop}

We get the following corollary.

\begin{prop}\label{prop:piecewise_cts}
Let $b\in [1/2, 1)$. For $m\in \N$ with $m\geq 2$, $\Phi_m$ is piecewise linear and continuous. For $k\in \N$ with $k\geq 2$, The function $\Pi^n_k$ is piecewise linear and continuous. 
\end{prop}

\begin{proof}
By Proposition~\ref{prop:slopes}, $\pi_k$ and $\phi$ are piecewise linear functions. By Proposition~\ref{obs:formula} and since piecewise linear continuous functions are preserved under composition, the result follows by induction.\qed
\end{proof}

\subsection{$[\pi_k]_b$ and $[\Phi_{n-1}]_{b\mathbf{1}_{n-1}}$ are nondecreasing}\label{sec:nondecreasing}

%We will prove that $\Pi_k$ is a facet for $R_{b1_n}(\R^n, \Z^n)$ by verifying that the above hypotheses hold for $\pi_k$ and $\Phi_{n-1}$ in the following props.

For $k\in \N$ with $k\geq 2$, $[\pi_k]_b$ is nondecreasing by Proposition~\ref{prop:pik_nondecreasing}. We analyze $[\Phi_{m}]_{b\mathbf{1}_m}$.

\begin{prop}\label{prop:nondecreasing} Let $b\in [1/2, 1)$. Then $[\Phi_{m}]_{b\mathbf{1}_m}$ is nondecreasing for every $m\in \N$. %{\red OMIT? Specifically,  $[\Phi_{n-1}]_{b\mathbf{1}_{n-1}}$ is nondecreasing.}
\end{prop}

\begin{proof}
We prove it by induction on $m$. For $m=1$, since $\phi=\pi_2$, it follows that $[\phi]_b$ is nondecreasing. Assume that $[\Phi_{m-1}]_{b\mathbf{1}_{m-1}}$ is nondecreasing and consider $[\Phi_{m}]_{b\mathbf{1}_m}$. Let $(x^1, x^2), (y^1, y^2)\in \R\times \R^{m-1}$ be such that $(x^1, x^2)\leq (y^1, y^2)$. Recall that $\Phi_m = \phi \diamond \Phi_{m-1} := [\psi]^{-1}_{(b,b\mathbf{1}_{m-1})}$ where $\psi(z_1, z_2) = [\phi]_b(z_1+[\Phi_{m-1}]_{b\mathbf{1}_{m-1}}(z_2))$. As shown in Remark~\ref{rem:diff-seq-merge}, $\psi$ is pseudo-periodic since $\phi$ and $\Phi_{m-1}$ are minimal by Proposition~\ref{prop:Phi_m_minimal}. Therefore, applying Proposition~\ref{obs:represent} $(iii)$,
\begin{align*}
[\Phi_{m}]_{b\mathbf{1}_m}(x_1, x_2) & = [\phi]_b(x_1+[\Phi_{m-1}]_{b\mathbf{1}_{m-1}}(x_2)) \\
& \leq [\phi]_b(y_1+[\Phi_{m-1}]_{b\mathbf{1}_{m-1}}(y_2)) && \\%\text{since }[\phi]_b, [\Phi_{m-1}]_{b\mathbf{1}_{m-1}}\text{ are nondecreasing}\\
& = [\Phi_{m}]_{b\mathbf{1}_m}(y_1, y_2),
\end{align*}
where the inequality holds because $[\phi]_b, [\Phi_{m-1}]_{b\mathbf{1}_{m-1}}$ are nondecreasing. Thus $[\Phi_{m}]_{b\mathbf{1}_m}$ is nondecreasing.\qed

%Applying the definition of $\Phi_m = \phi \diamond \Phi_{m-1}$, we see that
%\begin{align*}
%& [\Phi_{m}]_{b\mathbf{1}_m}(x^1, x^2)\\
%= &x^1+\sum_{i=1}^{m-1}x^2_i-mb\Phi_m(x^1, x^2)\\
%%= &x^1+\sum_{i=1}^{m-1}x^2_i-mb\left(\bigg[[\phi]_b({\color{red}z^1 +[\Phi_{m-1}]_{b\chf_{m-1}}(\z^2)}\bigg]_{b\chf_m}^{-1}(x^1, x^2)\right)\\
%= &x^1+\sum_{i=1}^{m-1}x^2_i-mb\left(\frac{x^1+\sum_{i=1}^{m-1}(x^2_i)-[\phi]_b\bigg(x^1 +[\Phi_{m-1}]_{b\chf_{m-1}}(x^2)\bigg)}{mb}\right)\\
%= &[\phi]_b\bigg(x^1 +[\Phi_{m-1}]_{b\chf_{m-1}}(x^2)\bigg)\\
%\leq &[\phi]_b\bigg(y^1-\lfloor y^1\rfloor +[\Phi_{m-1}]_{b\chf_{m-1}}(\y^2)\bigg),
%\end{align*}
%{\color{red}(JOE: Is the red expression above too confusing? It is supposed to refer to the function $\psi$ defined in the definition of the sequential-merge.)}where the inequality follows since $[\phi]_b$ and $[\Phi_{m-1}]_{b\mathbf{1}_{m-1}}$ are nondecreasing and $x\leq \y$. Using the same substitutions as in the latter string of expressions, we see that $[\phi]_b\bigg(y^1-\lfloor y^1\rfloor +[\Phi_{m-1}]_{b\chf_{m-1}}(\y^2-\lfloor \y^2\rfloor)\bigg) = [\Phi_{m}]_{b\mathbf{1}_m}(y_1, y_2)$. Thus, $[\Phi_{m}]_{b\mathbf{1}_m}$ is nondecreasing.
%{\color{red}JOE: I believe that this proof is the only proof where our definitions of sequneital-merge/lifting-space/group-space representations come into play.}\qed
\end{proof}

\subsection{$E(\pi_k)$ and $E(\Phi_{m})$ are unique up to scaling}\label{sec:unique-upto-scaling}

\begin{prop}\label{prop:equal_sets} Let $b\in [1/2, 1)$. For $m\in \N$, the sets $E(\pi_k)$ and $E(\Phi_{m})$ are unique up to scaling. 
\end{prop}

\begin{proof} 
First, we consider $\pi_k$. If $k=2$, then by Remark~\ref{remark:GMI_upts} we have that $E(\pi_k)$ is unique up to scaling. So let $k \geq 3$ and let $\xi:\R\to \R_+$ be a continuous function such that $E(\pi_k)\subseteq E(\xi) $. We claim that $\xi = \xi(b)\pi_k$.

If $\xi(b)=0$, then $\xi(x) +\xi(b-x)=0$ for each $x\in \R$ since $E(\xi)\supseteq E(\pi_k)$. As $\xi$ is nonnegative, this implies that $\xi(x)=0$ for each $x\in \R$ and so $\xi = 0 \pi_k$. Now suppose that $\xi(b)\neq 0$. It is sufficient to show that the function defined pointwise by $\tilde{\xi}(x) := \frac{1}{\xi(b)}\xi(-x)$ is equal to the function defined by $\tilde{\pi}_k(x) :=\pi_k(-x)$. Recall that $\tilde{\pi}_k$ is extreme as discussed after Remark~\ref{rem:diff-seq-merge}. 

Observe that $E(\tilde{\pi}_k)\subseteq E(\tilde{\xi})$. Indeed, if $(x,y)\in E(\tilde{\pi}_k)$, then $\tilde{\pi}_k(x+y) = \tilde{\pi}_k(x)+\tilde{\pi}_k(y)$, and $\pi_k(-x-y) = \pi_k(-x)+{\pi}_k(-y)$ then follows from the definition of $\tilde{\pi}_k$. Since $E(\pi_k)\subseteq E(\xi)$, this implies $\tilde{\xi}(x+y) = \frac{1}{\xi(b)}\xi(-x-y) = \frac{1}{\xi(b)}\xi(-x)+\frac{1}{\xi(b)}\xi(-y) = \tilde{\xi}(x)+\tilde{\xi}(y)$. Hence $E(\tilde{\pi}_k)\subseteq E(\tilde{\xi})$. This observation has a few implications. First, it implies $\tilde{\xi}(0)+\tilde{\xi}(0) = \tilde{\xi}(0)$ and so $\tilde{\xi}(0)=0$, and also that $\tilde{xi}(b)=1$. Next, since $\tilde{\pi}_k$ is periodic, Proposition~\ref{obs:periodicity} implies that $\tilde \xi$ is periodic. Finally, since $\tilde{\xi}$ is continuous and periodic, it is locally bounded.

Using $\pi = \tilde{\pi}_k$ and $\theta = \tilde \xi$ in Lemma~\ref{lem:main-lem}, it follows that $\tilde\xi = \tilde{\pi}_k$ on $I^k_3\cup I^k_6$. From Proposition~\ref{claim:ik2_induction_general} and again setting $\pi = \tilde{\pi}_k$ and $\theta = \tilde\xi$, we obtain that $\tilde\xi = \tilde{\pi}_k$ on $I_2^k\cup I_4^k$. It is left to show that $\tilde\xi = \tilde{\pi}_k$ on $I^k_1$ and $I^k_5$.

Let $U =  \left[0,\frac{b}{2}\left(\frac{1}{8}\right)^{k-2}\right]$ and observe that $U+U  = \left[0, b\left(\frac{1}{8}\right)^{k-2}\right]=I^k_1$. It follows from the definition of $\tilde{\pi}_k$ that $\tilde{\pi}_k(x)+\tilde{\pi}_k(y)=\tilde{\pi}_k(x+y)$ for all $x,y, x+y\in I^k_1$, so $U\times U \subseteq E(\tilde{\pi}_k)\subseteq E(\tilde\xi)$. Recall that $\tilde\xi(0) = \tilde{\pi}_k(0)=0$. Also, since $\tilde\xi = \tilde{\pi}_k$ on $I_2^k$, $\tilde\xi\left(b\left(\frac{1}{8}\right)^{k-2}\right) = \tilde{\pi}_k\left(b\left(\frac{1}{8}\right)^{k-2}\right)$. Thus $\tilde\xi = \tilde{\pi}_k$ on the endpoints of $I^k_1$. Moreover, Lemma~\ref{lem:interval-lemma} implies that $\tilde\xi$ is affine over $I_1^k$. Since $\tilde{\pi}_k$ is also affine over $I_1^k$ and $\tilde\xi = \tilde{\pi}_k$ at the endpoints, we have $\tilde\xi = \tilde{\pi}_k$ on $I_1^k$. The fact that $\tilde\xi = \tilde{\pi}_k$ on $I^k_5$ follows by symmetry (note that $\tilde\xi$ is also symmetric because $E(\tilde{\pi}_k) \subseteq E(\tilde\xi)$). Therefore $\tilde\xi = \tilde{\pi}_k$ everywhere.

\smallskip
Now consider $\Phi_m$. Dey and Richard's proof of Theorem~\ref{thm:dey-richard-seq-merge} shows that if $E(f)$ and $E(g)$ are unique up to scaling, then $E(f\diamond g)$ is also unique up to scaling. If $m=1$, then $\Phi_m = \phi$. Since $\phi=\pi_2$, the set $E(\phi)$ is unique up to scaling by Remark~\ref{remark:GMI_upts}. Now an induction argument shows that $E(\Phi_m)$ is unique up to scaling.\qed
%
%So suppose that $E(\Phi_{m-1})$ is unique up to scaling and consider $E(\Phi_m)$. Let $\Theta:\R^m\to \R_+$ be continuous such that $E(\Theta)\supseteq E(\Phi_m)$. We claim that $\Theta= \Theta(b1_m)\Phi_m$.
%
%If $\Theta(b1_m)=0$, then $\Theta(x)=0$ for all $x\in \R^m$ since $\Phi(x)+\Phi(b1_m-x) = \Phi(b1_m)=0$ and $\Phi$ is nonnegative. So assume that $\Theta(b1_m)>0$. It is sufficient to show that $\tilde{\Theta} := \frac{1}{\Theta(b1_m)}\Theta = \Phi_m$. Indeed, one can show this equality holds by {\red emulating????????? the referee will kill us}  the proofs of Theorem 5 and Proposition 11 in ~\cite{dey2} (note that Proposition 11 in ~\cite{dey2} requires validity of $\Theta$. However, the proof holds if $\Theta$ is assumed to be continuous). 
\end{proof}

\subsection{The proof of Theorem~\ref{thm:seq_merge}}

\begin{prop}\label{prop:PhiFacet} Let $b\in [1/2, 1)$. For $m\in \N$, the function $\Phi_m$ is a facet for $I_{b1_m}$. %{\red OMIT Specifically, $\Phi_{n-1}$ is a facet for $R_{b1_{n-1}}(\R^{n-1}, \Z^{n-1})$}.
\end{prop}

\begin{proof}
Using induction, the result is a consequence of Theorem~\ref{thm:dey-richard-seq-merge}; the assumptions of Theorem~\ref{thm:dey-richard-seq-merge} are verified by the results of Propositions \ref{prop:Phi_m_minimal}, \ref{prop:piecewise_cts}, \ref{prop:nondecreasing} and \ref{prop:equal_sets}.\qed
\end{proof}

The next few propositions argue that $\Pi^n_k$ is genuinely $n$ dimensional with at least $k$ slopes. Note that, unlike the one dimensional setting in which exactly $k$ slopes are attained, we are unsure of exactly how many slopes $\Pi^n_k$ attains; all we can establish is that the number of slopes is greater than or equal to $k$.
\begin{prop}\label{prop:0_on_integers}
Let $b\in [1/2, 1)$ and $m\in \N$. Then $\Phi_m(x)=0$ if and only if $x\in \Z^m$. Also, for every $n, k \in \N$ such that $n, k \geq 2$, $\Pi^n_k(x)=0$ if and only if $x\in \Z^n$. 
\end{prop}

\begin{proof}
We use induction on $m$. If $m=1$, then $\Phi_m =\phi$ and the result follows from~\eqref{Gom-funct}. So assume that $\Phi_t(x)=0$ if and only if $x\in \Z^t$ for all $t\leq m$ with $m,t\in \N$. Using the formulas in Proposition~\ref{obs:formula} and the induction hypothesis, it directly follows that $\Phi_{m+1}(x)=0$ for all $x\in \Z^{m+1}$. Let $x\in \R^{m+1}\setminus \Z^{m+1}$. By the induction hypothesis, if $x_{-1}\not\in \Z^m$, then $\Phi_m(x_{-1})>0$, and since $\phi$ is nonnegative, $\Phi_{m+1}(x)>0$ follows from Proposition~\ref{obs:formula}. If $x_{-1}\in \Z^m$, then $\Phi_m(x_{-1})=0$ by the induction hypothesis, and $\sum_{i=1}^{m+1}x_i-mb\Phi_m(x_{-1}) = \sum_{i=1}^{m+1}x_i\not\in \Z$. Again using the induction hypothesis, $\phi(\sum_{i=1}^{m+1}x_i-mb\Phi_m(x_{-1}) ) >0$ and so $\Phi_{m+1}(x)>0$ using the formula in Proposition~\ref{obs:formula}. 

For $\Pi^n_k$ the result follows by applying the same argument as above and noting that $\pi_k(x)=0$ if and only if $x\in \Z$; see Proposition~\ref{prop:nonnegativity}.
\qed
\end{proof}

\begin{prop}\label{prop:genuine} Let $b\in [1/2, 1)$. The function $\Pi^n_k$ is genuinely $n$ dimensional for every $n,k \in \N$ such that $n,k \geq 2$.
\end{prop}

\begin{proof} Assume to the contrary that $\Pi^n_k$ is not genuinely $n$ dimensional. Then there exist a linear transformation $T:\R^n\to \R^{n-1}$ and a function $\Psi:\R^{n-1}\to \R$ such that $\Pi^n_k = \Psi\circ T$. Since $T$ is linear with nontrivial kernel, there must exist $x \in \ker(T)$ such that $x \not\in\Z^n$. It follows that
$$\Pi^n_k(x) = \Psi\circ T(x) = \Psi(0) = \Psi\circ T(0) = \Pi^n_k(0) = 0.$$
However, Proposition~\ref{prop:0_on_integers} implies that $x\in \Z^n$, which is a contradiction.
%Proposition~\ref{obs:formula} states that
%\begin{equation*}
%0 = \Pi_k(v_1, v_2, \dots, v_n) = \frac{(n-1)\Phi_{n-1}(v_{-1})+\pi_k\bigg(\sum_{i=1}^nv_i-(n-1)b\Phi_{n-1}(v_{-1})\bigg)}{n},
%\end{equation*}
%which implies that
%$$(n-1)\Phi_{n-1}(v_{-1}) = -\pi_k\bigg(\sum_{i=1}^nv_i-(n-1)b\Phi_{n-1}(v_{-1})\bigg).$$
%The left hand side is nonnegative and the right hand side is nonpositive, indicating that both expressions are 0. Since $\Phi_{n-1}$ is only 0 at $\Z^{n-1}$, $(v_{-1}) \in \Z^{n-1}$. Substituting this into the right hand side, we see that $\pi_k(v_1 + z) = 0$ where $z = v_2 + \ldots + v_n \in \Z$, implying that $v_1\in \Z$. Hence $v\in \Z^n$, which is a contradiction. So $\Pi_k$ is genuinely $n$-dimensional.
\qed
\end{proof}

\begin{lemma}\label{lemma:affine_on_b}
Let $b\in [1/2, 1)$ and $m\in \N$. Then $\Phi_m(x) = \frac{1}{mb}\sum_{i=1}^m x_i$ for all $x\in \R^m_+$ with $\|x\|_{\infty}< b$. 
\end{lemma}
\begin{proof}
We proceed by induction on $m$. If $m=1$, then $\Phi_m = \phi$ and the result follows by the definition of the GMI in~\eqref{Gom-funct}. So let $m\geq 2$ and assume that $\Phi_{m-1}(x) = \frac{1}{(m-1)b}\sum_{i=1}^{m-1} x_i$ for all $x\in \R^{m-1}_+$ with $\|x\|_{\infty}< b$.

Let $x\in \R^m_+$ with $\|x\|_{\infty}< b$. Using Proposition~\ref{obs:formula} and the induction hypothesis, we see that
\begin{align*}
\Phi_m(x) & = \frac{ (m-1)\Phi_{m-1}(x_{-1})+\phi\bigg(\sum_{i=1}^mx_i-(m-1)b\Phi_{m-1}(x_{-1})\bigg)}{m}\\
& = \frac{ \frac{1}{b}\sum_{i=2}^mx_i+\phi(x_1)}{m}.
\end{align*}
Since $|x_1|<b$, we can apply the definition of the GMI to the previous equality and see
$$\Phi_m(x) = \frac{ \frac{1}{b}\sum_{i=2}^mx_i+\frac{1}{b}x_1}{m} = \frac{1}{mb}\sum_{i=1}^m x_i,$$
as desired. \qed\end{proof}

\begin{prop}\label{prop:atleastkslopes} Let $b\in [1/2, 1)$. The function $\Pi^n_k$ has at least $k$ slopes for every $n,k \in \N$ such that $n,k \geq 2$.
\end{prop}

\begin{proof}
By Theorem~\ref{thm:kslopes}, $\pi_k$ has $k$ nondegenerate intervals $J_1, \dots, J_k\subseteq \R$ such that $\pi_k$ is affine over each $J_i$ with slope $\sigma_i \in \R$, i.e., $\pi_k(x) = \sigma_ix + d_i$ for some $d_i \in \R$. Moreover, $\sigma_i\neq \sigma_j$ for $i\neq j$. For each $i=1,\dots, k$, let $R_i\subseteq \R^n$ be defined by 
$$R_i:= \{x\in \R^n: x_1\in J_i, ~x_{-1}\in B_{n-1}\},$$
where $B_{n-1} = \{x\in \R^{n-1}_+: \|x\|_{\infty} < b\}$. We claim that $\Pi^n_k$ is affine over each $R_i$, and attains a different slope on each $R_i$.

In order to see that $\Pi^n_k$ is affine over $R_i$, let $x\in R_i$. By Proposition~\ref{obs:formula}, we have \begin{align*}
\Pi^n_k(x) &= \frac{ (n-1)\Phi_{n-1}(x_{-1})+\pi_k(\sum_{i=1}^nx_i-(n-1)b\Phi_{n-1}(x_{-1}))}{n} && \\%\text{by Proposition~\ref{obs:formula}}\\
& = \frac{\frac{1}{b}\sum_{i=2}^nx_i+\pi_k(x_1)}{n} \quad\qquad \text{by Lemma~\ref{lemma:affine_on_b}}&& \\
& = \frac{1}{bn}\sum_{i=2}^nx_i+\frac{\sigma_ix_1 + d_i}{n} \;\qquad \text{since }x\in R_i && \\
& = \left(\frac{\sigma_i}{n}, \frac{1}{bn}\chf_{n-1}\right)\cdot x + \frac{d_i}{n}.
\end{align*}
Thus, $\Pi^n_k(x)$ is affine over $R_i$ with gradient $(\frac{\sigma_i}{n}, \frac{1}{bn}\chf_{n-1})$. 

Since each $\sigma_i$ is distinct for $i=1,\dots, n$, each gradient $(\frac{\sigma_i}{n}, \frac{1}{bn}\chf_{n-1})$ is distinct. Note that as $R_i$ is full dimensional, this vector is indeed a gradient. Hence, $\Pi^n_k$ has at least $k$ slopes, as desired. \qed\end{proof}

\begin{proof}[of Theorem~\ref{thm:seq_merge}] Since facets are periodic with respect to $\Z^n$, we may assume that $b\in (0,1)$. First, we prove the result for $b\in [1/2, 1)$. Sections~\ref{sec:minimal-bb}, \ref{sec:pwl-cont}, \ref{sec:nondecreasing}, \ref{sec:unique-upto-scaling}, and Propositions~\ref{prop:pi_k_facet} and \ref{prop:PhiFacet} establish that $\pi_k$ and $\Phi_{n-1}$ satisfy the assumptions for Theorem~\ref{thm:dey-richard-seq-merge}. Thus, $\Pi^n_k$ is a facet for $I_{b1_{n}}$. Proposition~\ref{prop:genuine} shows that $\Pi^n_k$ is genuinely $n$ dimensional, and Proposition~\ref{prop:atleastkslopes} shows that $\Pi^n_k$ has at least $k$ slopes. This gives the desired result. 

Now, let $b\in (0,1/2]$. By Theorem~\ref{thm:reflection} and the previous case of $b\in [1/2, 1)$, we obtain the desired result.
\qed\end{proof}

\section*{Acknowledgements}
We would like to thank Santanu Dey for very helpful discussions about the sequential-merge operation. We are also grateful to an anonymous reviewer for very insightful comments, which helped to shorten many proofs and catch errors in a previous draft. %\end{acknowledgements}

\appendix
\section{Appendix}
%{\red we have changed $\pi_b$ into $\pi$ and  $\pi_{1-b}$ into $\tilde \pi$.}

\begin{theorem}\label{thm:reflection}
Let $n\geq 1$ be a natural number. A function ${ \theta}:\R^n\to \R_+$ is minimal/extreme/facet for $I_b$ when $b\in [0,1/2]^n\setminus\{0\}$ if and only if ${ \tilde\theta}:\R^n\to \R_+$ defined by ${ \tilde\theta}(x) := { \theta}(-x)$ is minimal/extreme/facet for $I_{\chf-b}$, respectively, where $\chf\in \R^n$ is the vector of all ones. \end{theorem}

\begin{proof} The result essentially follows by applying Theorem 8.2 in~\cite{johnson1974group}. However, since that paper considers the so-called mixed-integer problem, while we are looking at the pure integer problem, we include a proof for completeness.

We show one direction as the other follows from swapping the roles of $b$ and $\chf-b$. Suppose that ${ \theta}$ is minimal for $I_b$ with $b\in [0,1/2]^n\setminus\{0\}$. Define ${ \tilde\theta}(x):= { \theta}(-x)$. We check that ${ \tilde\theta}$ is minimal using Theorem~\ref{thm:minimalinteger}. 
%Observe that ${ \tilde\theta}$ is nonnegative since ${ \theta}$ is. 
If $w\in \Z^n$ then so is $-w$, and therefore ${ \tilde\theta}(w) = { \theta}(-w) = 0$ since ${ \theta}$ is minimal. Let $x, y\in \R^n$ and note that
$${ \tilde\theta}(x+y) = { \theta}(-x-y) \leq { \theta}(-x)+{ \theta}(-y)  = { \tilde\theta}(x)+{ \tilde\theta}(y),$$
where the inequality follows from the subadditivity of ${ \theta}$. Hence ${ \tilde\theta}$ is subadditive. Finally, let $r\in \R^n$ and note that
$${ \tilde\theta}(r)+{ \tilde\theta}((\chf-b)-r) = { \theta}(-r)+{ \theta}(b-(\chf-r)) =  { \theta}(-r)+{ \theta}(b-(-r)) = 1,$$
where the second equation follows from the periodicity of ${ \theta}$ and the third equation from the symmetry of ${ \theta}$. Hence ${ \tilde\theta}$ is symmetric about $\chf-b$. From Theorem~\ref{thm:minimalinteger}, ${ \tilde\theta}$ is minimal.

\smallskip

Now assume that ${ \theta}$ is extreme. Let $\theta_1, \theta_2$ be valid for $I_{\chf-b}$ such that ${ \tilde\theta} = \frac{\theta_1+\theta_2}{2}$. We claim that $\tilde{\theta}_i(r):=\theta_i(-r)$, $i=1,2$, is a valid function for $I_b$. This would imply ${ \tilde\theta} = \theta_1=\theta_2$ from the extremality of ${ \theta}$. Let $y\in I_b$. Then $\tilde{y}(r) := y(-r)\in I_{\chf-b}$. Note that for $i=1,2$,
$$\sum_{r\in \R^n}\tilde{\theta}_i(r)y(r) = \sum_{r\in \R^n}\theta_i(-r)y(r) =  \sum_{r\in \R^n}\theta_i(-r)\tilde{y}(-r) \geq 1,$$
since $\theta_i$ is valid for $I_{\chf-b}$.

\smallskip

The proof that ${ \theta}$ is a facet if and only if ${ \tilde\theta}$  is a facet is similar.
\qed
\end{proof}

\begin{proof}[of Proposition~\ref{prop:nonnegativity}] We prove this using induction on $k$. For $k=2$, the result is easily verified using~\eqref{Gom-funct}. So let $k\geq 3$ and assume that $\pi_{k-1}$ is well-defined, nonnegative, and positive on $\R\setminus \Z$. First, we will show that $\pi_k$ is well-defined at the points $\{b\left(\frac{1}{8}\right)^{k-2}, 2b(\frac{1}{8})^{k-2}, b-2b(\frac{1}{8})^{k-2}, b-b(\frac{1}{8})^{k-2}, b\}$, that is, we will show $\pi_k$ is well-defined at the boundaries of the intervals on which it is defined. This will show that $\pi_k$ is well-defined on $[0,1)$, and since $\pi_k$ is periodic by definition, it will follow that $\pi_k$ is well-defined everywhere.

Note that
\begin{equation}\label{eq:well_defined_1}
\frac{4^{2-k}}{1-b}-\left(\frac{1}{1-b}\right)b\left(\frac{1}{8}\right)^{k-2}  = \left(\frac{2^{k-2}-b}{b-b^2}\right)b\left(\frac{1}{8}\right)^{k-2} =  \left(\frac{2^{k-2}-b}{1-b}\right)\left(\frac{1}{8}\right)^{k-2}>0, 
\end{equation}
where the inequality follows since $k\geq 3$ and $b\in (0,1)$. Also, observe that
\begin{align}\label{eq:well_defined_2}
\frac{1-4^{2-k}}{1-b}-\left(\frac{1}{1-b}\right)\left(b-b\left(\frac{1}{8}\right)^{k-2}\right)  &= \frac{1-2^{k-2}}{1-b}+\left(\frac{2^{k-2}-b}{b-b^2}\right)\left(b-b\left(\frac{1}{8}\right)^{k-2}\right)\nonumber \\
& = \frac{1+b((\frac{1}{8})^{k-2}-1)-4^{2-k}}{1-b}\nonumber\\
& \geq \frac{1+\frac{1}{2}((\frac{1}{8})^{k-2}-1)-4^{2-k}}{1-b}\nonumber\\
& > 0,
\end{align}
where the inequalities follow since $b\in (0,\frac{1}{2}]$ and $k\geq 3$. Equations~\eqref{eq:well_defined_1} and~\eqref{eq:well_defined_2} show that $\pi_k$ is well-defined and positive at the points $b\left(\frac{1}{8}\right)^{k-2}$ and $b-b\left(\frac{1}{8}\right)^{k-2}$. 

Observe that $b\in I^j_5\cap I^j_6$. Since $\frac{1-2^{k-2}}{1-b}+\left(\frac{2^{k-2}-b}{b-b^2}\right)b = 1$, it follows that $\pi_k$ is well-defined and positive at $b$.

Notice that $2b\left(\frac{1}{8}\right)^{k-2}\in I^k_2\subseteq I^{k-1}_1$ by definition. Similarly, $b-2b\left(\frac{1}{8}\right)^{k-2}\in I^k_4\subseteq I^{k-1}_5$. Therefore, by induction,
\begin{equation}\label{eq:well_defined_4}
\frac{4^{2-k}}{1-b}-\left(\frac{1}{1-b}\right)2b\left(\frac{1}{8}\right)^{k-2} =  \pi_{k-1}\left(2b\left(\frac{1}{8}\right)^{k-2}\right)>0, 
\end{equation}
and
\begin{equation}\label{eq:well_defined_5}
\frac{1-4^{2-k}}{1-b}-\left(\frac{1}{1-b}\right)\left(b-2b\left(\frac{1}{8}\right)^{k-2}\right)=  \pi_{k-1}\left(b-2b\left(\frac{1}{8}\right)^{k-2}\right)>0, 
\end{equation}
where the inequalities follows since $\pi_{k-1}$ is nonnegative. Thus, $\pi_k$ is well-defined and positive on $2b(\frac{1}{8})^{k-2}$ and $b-2b(\frac{1}{8})^{k-2}$.

Continuity of $\pi_k$ follows from the recursive piecewise definition and the confirmation above that the values are well-defined on the boundaries of the pieces.

We now show that $\pi_k$ is nonnegative and $\pi_k(x)=0$ if and only if $x\in \Z$. Let $x\in [0,1)$. If $x\in I^k_3\cup I^k_6$, then $\pi_k(x) = \pi_{k-1}(x)> 0$ by the induction hypothesis. Observe that $\pi_k$ is affine on the intervals $I^k_1, I^k_2, I^k_4,$ and $I^k_5$. From Equations~\eqref{eq:well_defined_1}-\eqref{eq:well_defined_5}, $\pi_k$ is positive on the endpoints of each of the latter intervals, except for when $x=0$. Thus, if $x\in I^k_1\setminus \{0\}\cup I^k_2\cup I^k_4\cup I^k_5$, then $\pi_k(x)> 0$ and $\pi_k(0)=0$. Finally, if $x\in \R\setminus [0,1)$, then by the periodicity of $\pi_k$, it follows that $\pi_k(x)> 0$ and $\pi_k(x)=0$ if $x\in \Z$.\qed
\end{proof}

%\bibliographystyle{plain}      % basic style, author-year citations
%\bibliography{full-bib}   % name your BibTeX data base

\end{document}